\documentclass[3p,times]{elsarticle}
\usepackage[colorlinks,linkcolor=blue,citecolor=blue]{hyperref}
\usepackage{bookmark}
\usepackage{bm}
\usepackage{amsmath,amsfonts,amsmath,amssymb,amsthm,cases}
\usepackage{algorithm}
\usepackage{algpseudocode} 
\usepackage{multirow}
\usepackage{longtable}
\usepackage{fullpage}
\usepackage{cleveref}
\usepackage{graphics,graphicx,epsfig,epstopdf,subfigure}
\usepackage{color,natbib}
\graphicspath{{jpg/}}

\newtheorem{theorem}{Theorem}[section]
\newtheorem{lemma}{Lemma}[section]

\newtheorem{example}{Example}[section]

\newtheorem{assumption}{Assumption}[section]
\numberwithin{equation}{section}

\newcommand{\dd}{\,{\rm d}}

\newcommand{\dist}{\,{\rm dist}}
\newcommand{\divo}{\,{\rm div}}

\begin{document}
	
\begin{frontmatter}
	\title{Optimal error estimates of the diffuse domain method for semilinear parabolic equations}
	\author[psu]{Yuejin Xu\corref{cor1}}
	\ead{yuejinxu@um.edu.mo}
	\address[um]{Faculty of Science and Technology, University of Macau, 999078, Macau}
	\cortext[cor1]{Corresponding author}
	\begin{abstract}
		In this paper, we mainly discuss the convergence behavior of diffuse domain method (DDM) for solving semilinear parabolic equations with Neumann boundary condition defined in general irregular domains. We use a phase-field function to approximate the irregular domain and when the interface thickness tends to zero, the phase-field function will converge to indicator function of the original domain. With this function, we can modify the problem to another one defined on a larger rectangular domain that contains the targer physical domain. Based on the weighted Sobolev spaces, we prove that when the interface thickness parameter goes to zero, the numerical solution will converge to the exact solution. Also, we derive the corresponding optimal error estimates under the weighted $L^2$ and $H^1$ norms. Some numerical experiments are also carried out to validate the theoretical results.
		
	\end{abstract}
\end{frontmatter}	
	
\section{Introduction}

Complex geometric shapes are ubiquitous in our realistic life and scientific problems arsing from biological systems and man-made object, such as vein networks in plant leaves, tumors in human bodies, microstructures in materials, etc. In recent years, interface problems on complex geometry have been received widespread attention. For example, the Allen-Cahn equation \cite{Acta1979} for describing the phase transition and separation \cite{DuFeng2020}, the time-dependent advection-diffusion equation and the Navier-Stokes equations for fluid dynamics \cite{Roger1984}, the Stokes-Darcy problems arising in petroleum and biomedical engineering \cite{MartinaBoris2023}. Relied on explicit surface parameterization, sharp interface approaches are always been applied to develop numerical methods for interface problems. The drawback of this approach is the complicated mesh generation. Quite a lot of approaches are used to solve interface problems, such as extended and composite finite element methods \cite{DolbowHarari2009,FriesBelytschko2010}, immersed interface methods \cite{LeVequeLi1995,LiIto2006,WangKai2023}, virtual node methods with embedded boundary conditions \cite{BedrossianZhu2010,HellrungLee2012}, and matched interface and boundary methods \cite{ZhaoWei2009,LiFeng2015,AmlanRay2023}, and so on. However, standard finite element and finite difference software packages can't be directly used to implement these methods. So, it's necessary to develop some corresponding software packages, which is time-consuming and costly. Compared with the traditional methods, in the past two decades, the diffuse interface method (DIM) has gained widespread attention due to its flexibility and applicativity. Instead of parameterizing the physical domain explicitly, DIM method represents the domain implicitly with a phase-field function, which can be regarded as an indicator of the domain when the interface thickness tends to zero. Then, this method can be coupled with some classical discretization schemes, such as finite difference method,  finite element method \cite{YangMao2019}, spectral method \cite{BuenoFenton2006}, and Nitsche's method \cite{NguyenStoter2018}, then some efficient numerical methods have been developed for solving interface problems, two-phase flow problems, and various material and engineering applications. Kockelkoren et al. \cite{KockelkorenLevine2003} were the first to apply the diffuse interface method to study diffusion inside a cell with zero Neumann boundary conditions. Based on DIM, many researchers have tried to develop some numerical methods to solve different two-phase flow problems, such as the coupled Navier-Stokes or Cahn-Hilliard systems \cite{FrigeriGrasselli2015,Abels2009}, the miscible fluids of different densities \cite{AbelsLengeler2014}, two-phase flows in complex geometries \cite{LiuChai2022}, compressible fluids \cite{FeireislRocca2010}, Stokes-Darcy coupled equations \cite{MartinaBoris2023}. Moreover, DIM can not only be applied to two-phase problems, but also problems involving more than two phases, by introducing additional labeling functions to distinguish different phases \cite{BrannickLiu2015}. The diffuse interface method has also been used to solve the material model with interfaces that can be advected or stretched \cite{TeigenLi2009,TorabiLowengrub2009}, the patient-specific human liver model based on MRI scans \cite{Stoter2017}, the PDEs in moving geometries \cite{ElliottStinner2011}, the variational inverse problems \cite{BurgerElvetun2015}, etc. 

The diffuse domain method (DDM) \cite{AndersonMcFadden1998,AndreasAxel2006} is another numerical method developed on the basis of DIM. The phase-field function will transits from 1 inside the original domain to 0 outside the domain rapidly with a narrow transition zone \cite{KarlLowengrub2015}. The original PDE problem then can be rewritten as a modified one defined on a larger rectangular domain, which can easily overcome the difficulty of generating meshes for complex domains. On the one hand, for practical application, the DDM has been used to solve PDEs in complex, stationary, or moving geometries with Dirichlet, Neumann, and Robin boundary conditions \cite{LiLowengrub2009}. Meanwhile, DDM can be coupled with an interface model to simulate two-phase fluid flows while preserving thermodynamic consistency \cite{GuoYu2021}. The DDM approach has also  employed to develop biomedical models, such as the chemotaxis-fluid diffuse-domain model for simulating bioconvection \cite{WangChertock2023}, a needle insertion model \cite{Jerg2020}, etc. On the other hand, for theoretical analysis, many papers have presented the asymptotic analysis of diffuse domain method. Li et al. have proposed several numerical schemes and verified that these schemes can converge asymptotically to the correct sharp interface problem \cite{LiLowengrub2009}. Also, they have proposed that different choices of boundary condition will lead to different numerical accuracy, and specific choices can improve the accuracy to second order \cite{KarlLowengrub2015, AlandLowengrub2010}. In addition to asymptotic analysis, some researchers have also discussed the error estimates in the $L^2$, $L^{\infty}$, and $H^1$-norms for elliptic equations \cite{FranzRoos2012, Schlottbom2016}. Moreover, Burger et al. have analyzed the approximation error in the extended regular region by regarding the phase-field function as weights and constructing the weighted Sobolev spaces \cite{BurgerElvetun2017}. Moreover, the convergence rate for the Stokes-Darcy coupled problem has been discussed on the original domain \cite{MartinaBoris2023}.

In the biomathematics application, we always need to deal with data from structural MRI or CT snapshots at a single time point. The Fisher-KPP PDE (shown below) is a fundamental model to describe the spatio-temporal evolution of the normalized tumor cell density $u$ in the irregular domain $D$ consisting of white matter (WM) and grey matter (GM) regions. The PDE, with Neumann boundary conditions, is given by \cite{ZhangEzhov2025}
\begin{equation*}
	\left\{\begin{split}
		&u_t=\nabla\cdot(A\nabla u)+\rho u(1-u), \quad && \bm x \in D,\\
		&Au\cdot\bm n=0,\quad && \bm x\in \partial D, 
	\end{split}\right.
\end{equation*}
where $\rho$ is the proliferation rate and $A(\bm x)$ is the diffusion tensor, $\bm n$ is the outward normal vector on the boundary $\partial D$. To solve this problem efficiently, in this paper, we will discuss how to apply diffuse domain method to solve general semilinear parabolic equation with Neumann boundary condition. To ensure the approximate solution obtained by the DDM method is meaningful, we will also analyze the approximation error. In the following, we will focus on the semilinear parabolic equation with Neumann boundary condition:
\begin{equation}
	\label{eq2-1}
	\left\{
	\begin{split}
		& u_t = \nabla\cdot(A\nabla u) + f(t,u), \quad && \bm x \in D,\quad 0 \leq t \leq T\\
		& u|_{t=0} = u_0, \quad && \bm x \in D,\\
		& (A\nabla u) \cdot \bm n = g(t,\bm x), \quad && \bm x \in \partial D, \quad 0 \leq t \leq T,
	\end{split}
	\right.
\end{equation}
where $D$ is a $d$-dimensional irregular bounded domain with Lipschitz continuous boundary ($d\geq 1$), $T>0$ is the duration time, $A(\bm x)>0$ is the diffusion coefficient, $u(t,\bm x)$ is the unknown function,   $f(t,u)$ is the nonlinear term of the underlying system, $u_0(\bm x)$ is the initial value, and $g(t,\bm x)$ is the Neumann boundary value. In the remaining part, we will derive the error estimate of the diffuse domain method in the weighted $L^2$ and weighted $H^1$ norms. The analysis techniques mainly follow the weighted Sobolev space-based framework developed in \cite{BurgerElvetun2017}, but also with some improvement. Our analysis successfully proposed a rigorous error estimate of diffuse domain method and provided a possibility of decoupling the interface thickness with the spatial and temporal mesh size for numerical methods based on DDM. Moreover, the optimal convergence rates are derived and have been improved to second order in weighted $L^2$ norm and first order in weighted $H^1$ norm, which are both half an order higher than the existing result \cite{BurgerElvetun2017}. We have also verified the theoretical results by several numerical experiments. To the best of our knowledge, this work presented in this paper is the first study on rigorous error analysis of the DDM for semilinear parabolic equations. 

The rest of the paper is organized as follows. In Section \ref{algorithm-description}, we first proposed the diffuse domain method for \eqref{eq2-1}. Since the framework of diffuse domain method is similar as \cite{HaoJu2025}, we will only give a brief introduction of the algorithm. Then, several preliminaries and lemmas for the weighted Sobolev spaces are given in Section \ref{pre}. The convergence of the diffuse domain solution to the original solution as the interface thickness parameter goes to zero are proved in  Section \ref{theory}, together with the corresponding optimal error estimates under the weighted $L^2$ and $H^1$ norms. In Section \ref{numerical},  some numerical experiments are carried out to validate the theoretical results and demonstrate the excellent performance of the diffuse domain method. Finally, some remarks are drawn in Section \ref{conclusion}.

\section{The diffuse domain method}\label{algorithm-description}

First of all, some standard notations are proposed for later provement. For a given open bounded Lipschitz domain $\Omega \subset \mathbb{R}^d$ and nonnegative integer $s$, denote $H^s(\Omega)$ as the standard Sobolev spaces on $\Omega$ with norm $\|\cdot\|_{s,\Omega}$ and semi-norm $|\cdot|_{s,\Omega}$, and the corresponding $L^2$-inner product is $(\cdot,\cdot)_{\Omega}$. The corresponding norm of space $H^s(\Omega)$ is $\|\cdot\|_{s,\Omega}$ and $\|v\|_{k,\infty,\Omega}={\rm ess}\sup_{|\bm \alpha|\leq k}\|D^{\bm \alpha}v\|_{L^{\infty}(\Omega)}$ for any function $v$ such that the right-hand side term makes sense, where $\bm \alpha=(\alpha_1, \cdots, \alpha_d)$ is a multi-index and $|\bm \alpha|=\alpha_1+\cdots+\alpha_d$. $H_0^s(\Omega)$ is the closure of $C_0^{\infty}(\Omega)$ with homogeneous Dirichlet boundary conditions. 
Moreover, given two quantities $a$ and $b$, $a \lesssim b$ is the abbreviation of $a \leq Cb$, where the hidden constant $C$ is positive and independent of the mesh size; $a\eqsim b$ is equivalent to $a\lesssim b \lesssim a$.


In the following part, we always suppose that $u_0\in H^2(D)$, $A\in H^1(D)$, $g\in L^2(\partial D)$. Then by the variational formula, we can reformulate \eqref{eq2-1} as: find $u \in H^1(D)$ and $u_t \in L^2(D)$ such that
\begin{equation}
	\label{eq2-2}
	\left\{\begin{split}
		&(u_t,v)+a(u,v)=\ell(v), \qquad \forall \,v \in H^1(D),\ 0 \leq t \leq T,\\
		&u(0,\bm x)=u_0(\bm x),
	\end{split}
	\right.
\end{equation}
where the bilinear operator $a(\cdot,\cdot)$ is symmetric positive and defined by
\begin{equation}
	\label{bilinear_operator}
	a(w,v)=\int_{D} A\nabla w\cdot\nabla v \dd \bm x, \qquad \forall\,w, v\in H^1(D),
\end{equation}
and the linear operator $\ell(\cdot)$ is defined by
\begin{equation}
	\label{linear_operator}
	\ell(v)=\int_{D} f(u)v\dd \bm x + \int_{\partial D} gv\dd \sigma,\qquad \forall\,  v\in H^1(D),
\end{equation}

Since the exact domain $D$ is irregular, it's difficult to obtain the above integrals directly. We turn to use the DDM to convert the exact integrals on irregular domains to the weighted integrals on regular domains \cite{BurgerElvetun2017}. Since we have introduced the algorithm in details in our previous work \cite{HaoJu2025}, we will only show the DDM method briefly with some minor revision.

First, let us introduce a signed distance function $d_D(\bm x)=\dist(\bm x,D)-\dist(\bm x,\mathbb{R}^n\setminus D)$, $\bm x \in \mathbb{R}^n$. Then, the domain $D$, the boundary $\partial D$ and the extended domain $D_{s}$ ($s \in (-\epsilon,\epsilon)$) can be reformulated as $D=\{\bm x\,|\,d_D(\bm x)<0 \}$, $\partial D=\{\bm x\,|\,d_D(\bm x)=0 \}$, $D_s=\{\bm x\,|\,d_D(\bm x)<s \}$, respectively. Next, we will denote $\epsilon>0$ as a small interface thickness parameter, and $S$ as a smooth monotonic transition function, such as the sigmoidal function. For simplicity, in this paper, we take
\begin{equation}
	S(s) =\tanh(3s).
\end{equation}

We can easily obtain that $S(\cdot/\epsilon)$ converges to the sign function as $\epsilon$ tends to zero, and hence, the phase-field function $$\omega^{\epsilon}(\bm x):=(1+\varphi^{\epsilon}(\bm x))/2$$  converges to the indicator function $\chi_D$ of $D$. Denote the $\epsilon$-tubular neighborhood of $\partial D$ by $\Gamma_{\epsilon}=D_{\epsilon}\setminus \overline{D_{-\epsilon}}$.	Obviously,
\begin{equation*}
	\omega_{\epsilon}=1, \quad \forall\ \bm x \in D_{-\epsilon}, \quad \omega_{\epsilon}\in (0, 1), \quad \forall \ \bm x \in \Gamma_{\epsilon}, \quad \omega_{\epsilon}=0, \quad \forall\ \bm x \in D_{\epsilon}^C.
\end{equation*}

In order to generate the spatial mesh  conveniently, one usually further fixes a larger rectangular domain $\Omega$ such that $D\subset D_{\epsilon} \subset \Omega$ in practice. Fig. \ref{fig1} describe the geometric relation among the original physical domain $D$, the $\epsilon$-extension domain $D_{\epsilon}$ and the covering rectangular domain $\Omega$. Instead of computing the integrals over irregular domain $D$ exactly, DDM method will use a weighted averaging of the integrals over regular domain $\Omega$. According to the definition of $\omega_{\epsilon}$, we can know that $\omega_{\epsilon}\approx 0$, so the integrals over $\Omega$ in fact will be converted to that over $D_{\epsilon}$. Therefore, for convenience, our theoretical analysis is only based on the $D_{\epsilon}$.

\begin{figure}[htbp]
	\centering
	\label{fig1}
	\centering
	\includegraphics[width = 140pt,height=140pt]{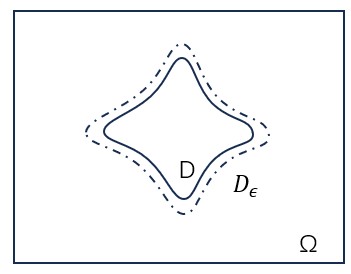}
	\caption{Sketch of an example geometry: $D \subset D_{\epsilon} \subset \Omega$ for some $\epsilon>0$.}
\end{figure}

Following the similar arguments of \cite{BurgerElvetun2017}, according to the Fubini's theorem and co-area formula, we can derive the approximation for the volume integral and boundary integral. For any integrable function $h:\Omega\to\mathbb{R}$,
\begin{subequations}
	\begin{align}
		\int_D h(\bm x)\dd\bm x&\approx\int_{D_{\epsilon}} h(\bm x)\omega_{\epsilon}(\bm x)\dd\bm x,\label{approx1}\\
		\int_{\partial D}h(\bm x)\dd\sigma(\bm x)&\approx \int_{D_{\epsilon}} h(\bm x)\left|\nabla \omega_{\epsilon}(\bm x) \right|\dd\bm x.\label{approx2}
	\end{align}
\end{subequations}

Let us define the weighted  $L^p$ space on $D_{\epsilon}$, $1\leq p < \infty$, associated with the phase-field function $\omega_{\epsilon}$ as follows:
\begin{equation*}
	L^p(D_{\epsilon};\omega_{\epsilon})=\left\{v\,\Big|\,\int_{D_{\epsilon}} |v|^p\,\omega_{\epsilon}\dd \bm x<\infty \right\},
\end{equation*}
with the norm
\begin{align*}
	\|v\|_{L^p(D_{\epsilon};\omega_{\epsilon})}=\left(\int_{D_{\epsilon}}|v|^p \,\omega_{\epsilon}\dd \omega_{\epsilon}\right)^{\frac 1 p}.
\end{align*}
Based on the weighted spaces $L^p(D_{\epsilon};\omega_{\epsilon})$,  the weighted Sobolev spaces are consequently defined as
\begin{equation*}
	W^{s,p}(D_{\epsilon};\omega_{\epsilon})=\left\{v\in L^p(D_{\epsilon};\omega_{\epsilon})\,|\, D^{\bm \alpha}v \in L^p(D_{\epsilon};\omega_{\epsilon}), \ \forall\, |\bm\alpha|\leq s \right\},
\end{equation*}
and $H^s(D_{\epsilon};\omega^{\epsilon}):=W^{s,2}(D_{\epsilon};\omega_{\epsilon})$ with the norm 
\begin{equation*}
	\|v\|_{W^{s,p}(D_{\epsilon};\omega_{\epsilon})}=\left(\int_{D_{\epsilon}}\sum\limits_{|\bm\alpha|\leq s}\left|D^{\bm \alpha}v\right|^p\dd\omega_{\epsilon}\right)^{\frac 1 p}.
\end{equation*}

Therefore, \eqref{approx1} and \eqref{approx2} leads to the following  diffuse domain approximation of \eqref{eq2-2} in $D_{\epsilon}$: find  $u^{\epsilon} \in L^2(0,T;H^1(D_{\epsilon};\omega^{\epsilon}))$ such that
\begin{equation}
	\label{weighted_variational}
	\left(u_t^{\epsilon},v \right)_{D_{\epsilon};\omega^{\epsilon}}+a^{\epsilon}(u^{\epsilon},v)= \ell^{\epsilon}(v), \qquad \forall \,v \in H^1(D_{\epsilon};\omega^{\epsilon}),
\end{equation}
where
\begin{equation*}
	\begin{split}
		a^{\epsilon}(w,v) &= \int_{D_{\epsilon}} \widetilde A\nabla w\nabla v \,\omega^{\epsilon}\dd \bm x,\qquad \forall\,w, v\in H^1(D_{\epsilon};\omega^{\epsilon}),\\
		\ell^{\epsilon}(v) &= \int_{D_{\epsilon}} \widetilde f(u^{\epsilon}) v\omega^{\epsilon}\dd \bm x + \int_{D_{\epsilon}} \widetilde gv\left|\nabla \omega^{\epsilon} \right|\dd \bm x,\qquad \forall\, v\in H^1(D_{\epsilon};\omega^{\epsilon})
	\end{split}
\end{equation*}
with the initial value $u^{\epsilon}(0) = \widetilde{u_0}$. Here $\widetilde A$, $\widetilde{u_0}$, $\widetilde f(u^{\epsilon})$  and $\widetilde g$ are certain extensions of $A$, $u_0$, $f(u)$  and $g$ from $D$ to $D_{\epsilon}$, respectively. 

Let us also assume the domain boundary $\partial D$ is of class $C^{1,1}$ from now on in this paper. It holds that $\dist(\bm x,\partial D)< \epsilon$ for all $\bm x \in \Gamma_{\epsilon}$ and $\dist(\bm x,\partial D)\geq \epsilon$ for all $\bm x \in \Omega\setminus \Gamma_{\epsilon}$.  Due to the regularity of $\partial D$, the projection of $\bm z \in \Gamma_{\epsilon}$ onto $\partial D$ is unique for  sufficiently small $\epsilon$ \cite{Adams1975}, i.e., for each $z \in \Gamma_{\epsilon}$, there exists a unique $\bm x \in \partial D$ such that $\bm z = \bm x+d_{D}(\bm z)\bm n(\bm x)$, where $\bm n(\bm x)$ is the outward unit normal vector for $\bm x\in \partial D$. The extension of function $f$, $u_0$ and $g$ are mainly based on the reflection with respect to the outward normal direction of the domain boundary $\partial D$, while keeping their corresponding regularity. The extension details can be found in the Chapter 4 of \cite{Adams1975} and Section 2 of \cite{HaoJu2025}.

Based on the definition of weight function $\omega_{\epsilon}$, it's obvious that the weighted function  $\omega_{\epsilon}(\bm x)$ and its gradient $\nabla \omega_{\epsilon}(\bm x)$ will vanish on $\partial D_{\epsilon}$ and outside of $D_{\epsilon}$,  we can further extend the integration on $D_{\epsilon}$ to the larger rectangular domain $\Omega$, which is easy for generating spatial grids in practice. Then, the original variational problem \eqref{eq2-2} can be further converted to: find  $u^{\Omega} \in L^2(0,T;H^1(\Omega;\omega^{\epsilon}))$ such that
\begin{equation}
	\label{weighted_variational_1}
	\left(u_t^{\Omega},v\right)_{\Omega;\omega^{\epsilon}}+a^{\Omega}(u^{\Omega},v)= \ell^{\Omega}(v), \quad \forall\, v \in H^1(\Omega;\omega^{\epsilon}), 
\end{equation}
where
\begin{equation*}
	\begin{split}
		a^{\Omega}(w,v) &= \int_{\Omega} \overline{A}\nabla w\nabla v \,\omega^{\epsilon}\dd \bm x,\quad \forall\,w, v\in H^1(\Omega;\omega^{\epsilon}),\\
		\ell^{\Omega}(v) &= \int_{\Omega} \overline{f}(u^{\Omega})v\,\omega^{\epsilon}\dd \bm x + \int_{\Omega} \overline gv\left|\nabla \omega^{\epsilon} \right|\dd \bm x,\quad \forall v\in H^1(\Omega;\omega^{\epsilon})
	\end{split}
\end{equation*}
with the initial value $u^{\Omega}(0)=\overline{u_0}$. Here $\overline{A}$, $\overline{u_0}$, $\overline f(t)$ and $\overline g(t)$ are  the extensions of $\widetilde{A}$, $\widetilde{u_0}(\bm x)$, $\widetilde f(t)$ and $\widetilde g(t)$ from $D_{\epsilon}$ to $\Omega$. It's worth to note that the function extensions in \eqref{weighted_variational_1} are not unique. Since $\omega_{\epsilon}(\bm x)=0$ for $\bm x \in \Omega\setminus D_{\epsilon}$ and consequently $u^{\epsilon}|_{D_{\epsilon}}=u^{\Omega}|_{D_{\epsilon}}$, we only need to focus on the error estimate between $u^\epsilon$ produced by \eqref{weighted_variational} and $u$ produced by \eqref{eq2-2}, which are defined on the domain $D_{\epsilon}$.

\section{Preliminaries}\label{pre}

Note that $\Gamma_{\epsilon}$ can be rewritten as
\begin{equation*}
	\Gamma_{\epsilon}=\left\{\bm z \in \Omega\,\big|\, \exists\  \bm x \in \partial D,\  |s|<\epsilon,\ \bm z =\bm x + s\bm n(\bm x) \right\}.
\end{equation*}
Furthermore, 
$\left|\Gamma_{\epsilon} \right|\lesssim \epsilon\mathcal{H}^{n-1}(\partial D)$,
where $\left|\Gamma_{\epsilon}\right|=\mathcal{L}^n(\Gamma_{\epsilon})$ is the $n$-dimensional Lebesgue measure of $\Gamma_{\epsilon}$ and $\mathcal{H}^{n-1}(\partial D)$ is the $(n-1)$-dimensional Hausdorff measure of $\partial D$. 

\subsection{On the weighted Sobolev space}

The following three theorems (Theorems \ref{trace_thm}, \ref{embed_thm} and \ref{Poincare_thm})
have readily been proved in \cite{BurgerElvetun2017}.
\begin{theorem}[Trace Theorem]
	\label{trace_thm}
	Let $\epsilon_0>0$ be sufficiently small and $1\leq p < \infty$. Then, there exists a constant $C>0$ such that for any $\epsilon \in [0, \epsilon_0]$ and $v \in W^{1,p}(D_{\epsilon};\omega_{\epsilon})$, there holds
	\begin{equation*}
		\int_{D_{\epsilon}} |v|^p|\nabla \omega_{\epsilon}|\dd \bm x \lesssim \|v\|_{W^{1,p}(D_{\epsilon};\omega_{\epsilon})}^p.
	\end{equation*}
\end{theorem}

\begin{theorem}[Embedding Theorem]\label{embed_thm}
	Suppose that $\epsilon \in (0, \epsilon_0]$, $1\leq p < \infty$, and $\alpha>0$ be the constant satisfies that for all $s \in (0, 2)$, $\zeta_1 s^{\alpha}\leq (1+S(s-1))/2\leq \zeta_2 s^{\alpha}$ for some constants $\zeta_1,\zeta_2>0$. Then the following embeddings are continuous
	\begin{equation*}
		W^{1,p}(D_{\epsilon};\omega_{\epsilon}) \hookrightarrow L^q(D_{\epsilon};\omega_{\epsilon}), \quad 1 \leq q \leq p_{\alpha}^*, \ q < \infty,
	\end{equation*}
	where
	$	p_{\alpha}^*=\frac{p(n+\alpha)}{n+\alpha-p}$ for $p < n+\alpha$ and  $p_{\alpha}^*=\infty$ for $p \geq n+\alpha.$
	Moreover, there exists a constant $C$ independent of $\epsilon$ such that for any $v \in W^{1,p}(D_{\epsilon};\omega_{\epsilon})$, there holds
	\begin{equation*}
		\|v\|_{L^q}(D_{\epsilon};\omega_{\epsilon}) \lesssim \|v\|_{W^{1,p}(D_{\epsilon};\omega_{\epsilon})}.
	\end{equation*}
\end{theorem}

\begin{theorem}[Poincare-Friedrichs-type inequality]\label{Poincare_thm}
	Suppose that $\epsilon \in (0,\epsilon_0]$, $1\leq p<\infty$, and  the extended domain $D_{\epsilon}$ be connected. Then, there exists a constant $C$ independent of $\epsilon$ such that for any $\epsilon \in (0,\epsilon_0)$ and $v \in W^{1,p}(D_{\epsilon};\omega_{\epsilon})$, there holds
	\begin{equation*}
		\|v\|_{L^p(D_{\epsilon};\omega_{\epsilon})} \lesssim \|\nabla v\|_{L^p(D_{\epsilon};\omega_{\epsilon})}^p + \int_{D_{\epsilon}}|v|^p|\nabla \omega_{\epsilon}|\dd \bm x.
	\end{equation*}
	
\end{theorem}

\subsection{On diffuse volume integrals}

The following two theorems readily come from Theorem 5.2 and Theorem 5.6 in \cite{BurgerElvetun2017}.

\begin{theorem}
	\label{thm1}
	Suppose that $\epsilon\in (0, \epsilon_0]$ and  $h(\bm x) \in H^1(D_{\epsilon};\omega_{\epsilon})$. Then, there exists a constant $C>0$ independent of $\epsilon$ such that
	\begin{equation*}
		\left|\int_{D_{\epsilon}} h(\bm x)\dd \omega_{\epsilon}(\bm x)-\int_D h(\bm x)\dd \bm x \right| \lesssim \epsilon^{\frac 3 2}\|h\|_{H^1(\Gamma_{\epsilon};\omega_{\epsilon})}.
	\end{equation*}
\end{theorem}

\begin{theorem}
	\label{thm2}
	Suppose that  $\epsilon\in (0, \epsilon_0]$, $w \in H^2(D_{\epsilon};\omega_{\epsilon})$ satisfying $w=0$ on $\partial D$ and  $v \in H^1(D_{\epsilon};\omega_{\epsilon})$. Then, there holds
	\begin{equation*}
		\int_{\Gamma_{\epsilon}} wv\left|\nabla \omega_{\epsilon}\right|\dd \bm x \lesssim \left(\epsilon^{\frac 3 2}\|w\|_{H^2(D_{\epsilon};\omega_{\epsilon})}+\epsilon^{\frac 3 2}\|w\|_{H^2(\Gamma_{\epsilon};\omega_{\epsilon})}\right)\|v\|_{H^1(D_{\epsilon};\omega_{\epsilon})},
	\end{equation*}
	where the hidden constant is independent of $\epsilon$, $u$ and $v$.
\end{theorem}

Next, let us introduce the smooth condition for  $f$ and some regularity condition required for the exact solution $u$ in order to carry out the convergence and error analysis of the diffuse domain method. For simplicity of expressions, in the following analysis we will not distinguish the exact solution $u(t)$, the source function $f(u)$, the Neumann  boundary value $g(t)$, the diffusion coefficient $A$ with their extensions  in $D_\epsilon$ since there are no ambiguities.

\begin{assumption}
	\label{mild_growth}
	The extension of the function $f(t,\zeta)$ to $D_{\epsilon}$ grows mildly with respect to $\zeta$, i.e., there exists a number $p>0$ for $d=1$, or $p\in (0, 2]$ for $d=2$, or $p\in (0,1]$ for $d=3$ such that
	\begin{equation}
		\label{mild_growth_eq}
		\left|\frac{\partial f}{\partial \zeta}(t,\zeta)\right|\lesssim 1+|\zeta|^p,\quad \forall\ t, \zeta \in \mathbb{R}.
	\end{equation}
\end{assumption}

\begin{assumption}
	\label{assumption1}
	The extension of the function $f(t,\bm x)$ to $D_{\epsilon}$ satisfies the following regularity condition:
	\begin{subequations}
		\begin{align}
			\label{smooth_cond1}
			\sup\limits_{0\leq t \leq T}\|f(t,\bm x)\|_{L^2(D_{\epsilon};\omega_{\epsilon})} &\lesssim 1,\\
			\label{smooth_cond2}
			\sup\limits_{0\leq t \leq T}\|f(t,\bm x)\|_{H^1(\Gamma_{\epsilon};\omega_{\epsilon})}&\lesssim 1,
		\end{align}
	\end{subequations}
	where the hidden constant may depend on the terminal time $T$.
\end{assumption}
\begin{assumption}
	\label{regularity_assump}
	The extension of the exact solution $u(t)$ of \eqref{eq2-2}   to $D_\epsilon$ satisfies the following regularity conditions:
	\begin{subequations}
		\begin{align}
			\label{regularity1}
			\sup\limits_{0\leq t \leq T}\|u(t)\|_{H^2(D_{\epsilon};\omega_{\epsilon})} &\lesssim 1,\\
			\label{regularity2}
			\sup\limits_{0\leq t \leq T}\|u_t(t)\|_{W^{1,\infty}(D_{\epsilon})}&\lesssim 1,
		\end{align}
	\end{subequations}
	where the hidden constants may depend on the terminal time  $T$.
\end{assumption}

The following lemma has been proved in Lemma 3.6 in \cite{}.
\begin{lemma}
	\label{lemma1}
	Suppose that  $0<\epsilon\leq \epsilon_0$, $w\in H^1(D_{\epsilon};\omega_{\epsilon})\cap L^{\infty}(D_{\epsilon})$, and $v\in H^1(D_{\epsilon};\omega_{\epsilon})$. Then
	\begin{equation*}
		\left|\int_{D_{\epsilon}}w v\omega_{\epsilon}\dd \bm x-\int_D w v\dd \bm x \right|\lesssim \epsilon^{\frac 1 2} \|w\|_{L^2(\Gamma_{\epsilon};\omega_{\epsilon})}\|v\|_{L^2(\Gamma_{\epsilon};\omega_{\epsilon})},
	\end{equation*}
	where the hidden constant  is independent of $\epsilon$, $u$ and $v$.
\end{lemma}

We then have the following lemma on the locally-Lipschitz continuity of $f(u)$.

\begin{lemma}
	\label{locally}
	Suppose that the function $f$ satisfies Assumption \ref{mild_growth}, and the exact solution $u\in L^(0,T;H^1(D))$ and its extension to $D_{\epsilon}$ fulfills \eqref{regularity1} in Assumption \ref{regularity_assump}. Then $f$ is locally-Lipschitz continuous with respect to weighted norms in a strip along the extended exact solution $u$, i.e., for any fixed constant $R>0$,
	\begin{equation}
		\label{locally_eq}
		\left\|f(t,v)-f(t,w)\right\|_{L^2(D_{\epsilon};\omega_{\epsilon})}\lesssim \left\|v-w\right\|_{H^1(D_{\epsilon};\omega_{\epsilon})},
	\end{equation}
	for ant $t \in [0,T]$ and $v,w\in H^1(D_{\epsilon};\omega_{\epsilon})$ satisfying
	\begin{equation}
		\label{locally_eq2}
		\max\left\{\left\|v-u\right\|_{H^1(D_{\epsilon};\omega_{\epsilon})},\left\|w-u\right\|_{H^1(D_{\epsilon};\omega_{\epsilon})} \right\}\leq R,
	\end{equation}
	where the hidden constant may depend on $R$.
	
\end{lemma}
\begin{proof}
	It follows from \eqref{regularity1} and \eqref{locally_eq2} that
	\begin{equation}
		\label{locally_eq3}
		\max\left\{\left\|v\right\|_{H^1(D_{\epsilon};\omega_{\epsilon})},\left\|w\right\|_{H^1(D_{\epsilon};\omega_{\epsilon})} \right\}\lesssim 1+R.
	\end{equation}
	By the Lagrange mean value theorem and \eqref{mild_growth_eq},
	\begin{equation*}
		\begin{aligned}
			&\left\|f(t,v)-f(t,w)\right\|_{L^2(D_{\epsilon};\omega_{\epsilon})}^2\\
			&\quad=\left\|\frac{\partial f}{\partial u}(t,\xi)(v-w)\right\|_{L^2(D_{\epsilon};\omega_{\epsilon})}^2\\
			&\quad\lesssim \int_{D_{\epsilon}} \left(1+|\xi|^p\right)^2|v-w|^2\omega_{\epsilon}\dd \bm x\\
			&\quad\leq \int_{D_{\epsilon}} \left(1+|v|^p\right)^2|v-w|^2\omega_{\epsilon}\dd\bm x+\int_{D_{\epsilon}}\left(1+|w|^p\right)^2|v-w|^2\omega_{\epsilon}\dd\bm x,
		\end{aligned}
	\end{equation*}
	where $\xi(\bm x)=\theta(\bm x)v(\bm x)+(1-\theta(\bm x))w(\bm x)$ for some $\theta(\bm x)\in [0,1]$. It's obvious that
	\begin{equation}
		\begin{aligned}
			\label{locally_eq4}
			\int_{D_{\epsilon}} \left(1+|v|^p\right)^2|v-w|^2\omega_{\epsilon}\dd\bm x=&\int_{D_{\epsilon}}|v|^{2p}|v-w|\omega_{\epsilon}\dd\bm x+2\int_{D_{\epsilon}}|v|^p|v-w|^2\omega_{\epsilon}\dd\bm x\\
			&+\int_{D_{\epsilon}}|v-w|^2\omega_{\epsilon}\dd\bm x.
		\end{aligned}
	\end{equation}
	
	It suffices to show below the bound of the first term in the right-hand side of \eqref{locally_eq4} since the bound of the other two terms can be derived similarly.
	
	Case I: $d=1$. In this case, $0<p<\infty$. We choose $q_1,q_2$ satisfying
	\begin{equation*}
		\frac{1}{q_1}+\frac{1}{q_2}=1,\quad 1\leq q_1<\infty,\quad 1\leq q_2<\infty, 2pq_1\geq 1.
	\end{equation*}
	By using Holder's inequality, embedding theorem for weighted Sobolev space (see Theorem \ref{embed_thm}) with $\alpha = 1$ and \eqref{locally_eq3}, we have
	\begin{equation}
		\begin{aligned}
			\label{locally_eq5}
			\int_{D_{\epsilon}}|v|^{2p}|v-w|^2\omega_{\epsilon}\dd\bm x &\leq\left(\int_{D_{\epsilon}}|v|^{2pq_1}\omega_{\epsilon}\dd\bm x\right)^{\frac{1}{q_1}}\left(\int_{D_{\epsilon}}|v-w|^{2q_2}\omega_{\epsilon}\dd\bm x\right)^{\frac{1}{q_2}}\\
			&=\left\|v\right\|_{L^{2pq_1}(D_{\epsilon};\omega_{\epsilon})}^{2p}\left\|v-w\right\|_{L^{2q_2}(D_{\epsilon};\omega_{\epsilon})}^2\\
			&\lesssim \|v\|_{H^1(D_{\epsilon};\omega_{\epsilon})}^{2p}\|v-w\|_{H^1(D_{\epsilon};\omega_{\epsilon})}^2\lesssim \|v-w\|_{H^1(D_{\epsilon};\omega_{\epsilon})}^2.
		\end{aligned}
	\end{equation}
	
	Case II: $d=2$. In this case, $0<p\leq 2$. According to the embedding theorem for weighted Sobolev space (see Theorem \ref{embed_thm}) with $\alpha=1$, we have $H^1(D_{\epsilon};\omega_{\epsilon})\hookrightarrow L^q(D_{\epsilon};\omega_{\epsilon})$, where $1\leq q\leq 6$. Select proper $q_1,q_2$ satisfying
	\begin{equation*}
		\frac{1}{q_1}+\frac{1}{q_2}=1,\quad 1\leq q_1<\infty,\quad 1\leq q_2\leq 3, \quad 1\leq 2pq_1\leq 6.
	\end{equation*}
	And with direct manipulation, we can know that these conditions hold if $(q_1,q_2)$ satisfies the conditions
	\begin{equation*}
		\max\left\{\frac 3 2,\frac{1}{2p} \right\}\leq q_1\leq \frac{3}{p},\quad 1 \leq q_2\leq 3.
	\end{equation*}
	Noting that $0<p\leq 2$, we can find the existence of such a pair $(q_1,q_2)$ in terms of the above conditions. Following the similar arguments for deriving \eqref{locally_eq5}, we have
	\begin{equation}
		\label{locally_eq6}
		\begin{aligned}
			\int_{D_{\epsilon}} |v|^{2p}|v-w|^2\omega_{\epsilon}\dd\bm x&\leq \left(\int_{D_{\epsilon}}|v-w|^{2q_2}\omega_{\epsilon}\dd\bm x\right)^{\frac{1}{q_2}}\left(\int_{D_{\epsilon}}|v|^{2pq_1}\omega_{\epsilon}\dd\bm x\right)^{\frac{1}{q_1}}\\
			&=\|v-w\|_{L^{2q_2}(D_{\epsilon};\omega_{\epsilon})}^2\|v\|_{L^{2pq_1}(D_{\epsilon};\omega_{\epsilon})}^{2p}\\
			&\lesssim\|v-w\|_{H^1(D_{\epsilon};\omega_{\epsilon})}^2\|v\|_{H^1(D_{\epsilon};\omega_{\epsilon})}^{2p}\lesssim \|v-w\|_{H^1(D_{\epsilon};\omega_{\epsilon})}^2.
		\end{aligned}
	\end{equation}
	
	Case III: $d=3$. In this case, $0<p\leq 1$. Following the similar derivation of Case II above, we have $H^1(D_{\epsilon};\omega)\hookrightarrow L^q(D_{\epsilon};\omega_{\epsilon})$, where $1\leq q \leq 4$. Select proper $q_1,q_2$ satisfying
	\begin{equation*}
		\frac{1}{q_1}+\frac{1}{q_2}=1,\quad \leq 1 \leq q_1<\infty,\quad 1\leq q_2\leq 2,\quad 1 \leq 2pq_1\leq 4.
	\end{equation*}
	It's obvious that these conditions hold if $(q_1,q_2)$ satisfies
	\begin{equation*}
		\max\left\{2,\frac{1}{2p} \right\}\leq q_1\leq \frac{2}{p},\quad 1 \leq q_2\leq 2.
	\end{equation*}
	Noting $0<p\leq 1$, we can always find a pair $(q_1,q_2)$ satisfies the above conditions and in the same way of \eqref{locally_eq6}
	\begin{equation}
		\label{locally_eq7}
		\int_{D_{\epsilon}}|v|^{2p}|v-w|^2\omega_{\epsilon}\dd\bm x\lesssim \|v-w\|_{H^1(D_{\epsilon};\omega_{\epsilon})}^2.
	\end{equation}
	
	Finally, the combination of \eqref{locally_eq4}-\eqref{locally_eq7} leads to \eqref{locally_eq}.

\end{proof}

\section{Error analysis}\label{theory}
To illustrate the convergence of diffuse domain method, let us combine the \eqref{eq2-2} and \eqref{weighted_variational} together, then we can get the error equation as follows: for all $v \in H^1(D_{\epsilon};\omega_{\epsilon})$,
\begin{equation}
	\label{eq4-1}
	\begin{split}
		&\left(u_t^{\epsilon}-u_t,v\right)_{D_{\epsilon};\omega_{\epsilon}}+a^{\epsilon}(u^{\epsilon}-u,v)\\
		&\quad=[\left(u_t,v\right)-\left(u_t,v\right)_{D_{\epsilon};\omega_{\epsilon}}]+[a(u,v)-a^{\epsilon}(u,v)]+[\ell^{\epsilon}(v)-\ell(v)].
	\end{split}
\end{equation}
Next we will derive the error estimates of $u^{\epsilon}-u$ in the $L^2$ and $H^1$ norms.

\subsection{Optimal error estimate in the $L^2$-norm}

\begin{theorem}[Error estimate in the $L^2$ norm]
	\label{thm3}
	Suppose that $0<\epsilon\leq \epsilon_0$,  $g\in H^2(D_{\epsilon};\omega_{\epsilon})$, $\kappa\leq A(\bm x)\leq \kappa^{-1}$ for all $\bm x \in D_{\epsilon}$ with some constant $\kappa>0$, and $f$ satisfies  Assumption \ref{mild_growth} and \ref{assumption1}. Assume the exact solution $u\in L^2(0,T;H^2(D))$ and its extension  to $D_\epsilon$ fulfills Assumption  \ref{regularity_assump}. Then we have
	\begin{equation}
		\label{thm3-1}
		\begin{split}
			\|u^{\epsilon}(t)-u(t)\|_{L^2(D_{\epsilon};\omega_{\epsilon})}\lesssim \epsilon^2, \quad \forall \,0\leq t \leq T,
		\end{split}
	\end{equation}
	where the hidden constant is independent of $\epsilon$.
\end{theorem}
\begin{proof}
	Let us analyze $\left(u_t,v\right)-\left(u_t,v\right)_{D_{\epsilon};\omega_{\epsilon}}$, $a(u,v)-a^{\epsilon}(u,v)$ and $\ell^{\epsilon}(v)-\ell(v)$ respectively for any $v \in H^1(D_{\epsilon};\omega_{\epsilon})$. 
	Recalling the definition of $a^{\epsilon}(\cdot,\cdot)$ and $a(\cdot,\cdot)$, since $\omega_{\epsilon}$ will vanish on  $\partial D_{\epsilon}$, we can derive the following formula by using the integration by part,
	{\small\begin{equation}\label{thm3-2}
			\begin{aligned}
				&a(u,v)-a^{\epsilon}(u,v)\\
				&\quad=\int_{D_{\epsilon}} \divo(A\nabla u)\omega_{\epsilon}v\dd\bm x+\int_{D_{\epsilon}} A\nabla u\nabla \omega_{\epsilon} v\dd \bm x+\int_{\partial D} \bm n\cdot \left(A\nabla u\right)v\dd \sigma-\int_D\divo(A\nabla u)v\dd \bm x\\
				&\quad=\int_{D_{\epsilon}}\divo\left(A\nabla u\right)\omega_{\epsilon} v\dd\bm x-\int_D\divo\left(A\nabla u\right)v\dd \bm x-\int_{D_{\epsilon}} \bm n\cdot A\nabla u\left|\nabla \omega_{\epsilon}\right|v\dd \bm x+\int_{\partial D} gv\dd \sigma,
			\end{aligned}
	\end{equation}}
	where in the last equality we use the fact that $\nabla \omega_{\epsilon}=-\bm n\left|\nabla \omega_{\epsilon}\right|$. By Theorem \ref{thm1}, we get
	\begin{align}
		\label{thm3-3}
		&\left|\int_D \divo(A\nabla u)v\dd \bm x-\int_{D_{\epsilon}} \divo(A\nabla u)v\omega_{\epsilon}\dd \bm x\right|
		\lesssim \epsilon^{\frac 3 2}\left\|\divo(A\nabla u)\right\|_{H^1(\Gamma_{\epsilon};\omega_{\epsilon})}\|v\|_{H^1(\Gamma_{\epsilon};\omega_{\epsilon})}.
	\end{align}
	Then, inserting \eqref{thm3-3} to \eqref{thm3-2}, we get that
	\begin{equation}
		\label{thm3-4}
		\begin{aligned}
			a(u,v)-a^{\epsilon}(u,v)\lesssim& \epsilon^{\frac 3 2}\left\|\divo(A\nabla u)\right\|_{H^1(\Gamma_{\epsilon};\omega_{\epsilon})}\|v\|_{H^1(\Gamma_{\epsilon};\omega_{\epsilon})}\\
			&-\int_{D_{\epsilon}} \bm n\cdot A\nabla u\left|\nabla \omega_{\epsilon}\right|v\dd \bm x+\int_{\partial D} gv\dd \sigma.
		\end{aligned}
	\end{equation}
	As for the estimate of $\left(u_t,v\right)-\left(u_t,v\right)_{D_{\epsilon};\omega_{\epsilon}}$, by using the conclusion of Theorem \ref{thm1}, we obtain
	\begin{equation}
		\label{thm3-5}
		\left(u_t,v\right)-\left(u_t,v\right)_{D_{\epsilon};\omega_{\epsilon}}\lesssim \epsilon^{\frac 3 2}\|u_t\|_{H^1(\Gamma_{\epsilon};\omega_{\epsilon})}\|v\|_{H^1(\Gamma_{\epsilon};\omega_{\epsilon})}.
	\end{equation}
	
	In the following, we first discuss the estimate of $\int_{D_{\epsilon}}f(u^{\epsilon})v\omega_{\epsilon}\dd\bm x-\int_D f(u)v\dd\bm x$. Since $f$ is locally-Lipschitz continuous (see Lemma \ref{locally}), we can derive
	\begin{equation}
		\label{thm3-6}
		\begin{aligned}
			&\int_{D_{\epsilon}} f(u^{\epsilon})v\omega_{\epsilon}\dd\bm x - \int_D f(u)v\dd\bm x\\
			&\quad=\int_{D_{\epsilon}} f(u^{\epsilon})v\omega_{\epsilon}\dd\bm x-\int_{D_{\epsilon}} f(u)v\omega_{\epsilon}\dd\bm x+\int_{D_{\epsilon}}f(u)v\omega_{\epsilon}\dd\bm x-\int_D f(u)v\dd\bm x\\
			&\quad\lesssim \left\|f(u^{\epsilon})-f(u)\right\|_{L^2(D_{\epsilon};\omega_{\epsilon})}\|v\|_{L^2(D_{\epsilon};\omega_{\epsilon})}+\epsilon^{\frac 3 2}\left\|f(u)\right\|_{H^1(\Gamma_{\epsilon};\omega_{\epsilon})}\|v\|_{H^1(\Gamma_{\epsilon};\omega_{\epsilon})}\\
			&\quad\lesssim\|u^{\epsilon}-u\|_{H^1(D_{\epsilon};\omega_{\epsilon})}\|v\|_{L^2(D_{\epsilon};\omega_{\epsilon})}+\epsilon^{\frac 3 2}\left\|f(u)\right\|_{H^1(\Gamma_{\epsilon};\omega_{\epsilon})}\|v\|_{H^1(\Gamma_{\epsilon};\omega_{\epsilon})}
		\end{aligned}
	\end{equation}
	where the last two inequality uses the conclusion of Theorem \ref{thm1}.
	
	With the estimate of \eqref{thm3-6}, we can derive the bound of $\ell^{\epsilon}(v)-\ell(v)$, 
	\begin{equation}
		\label{thm3-7}
		\begin{aligned}
			\ell^{\epsilon}(v)-\ell(v)=& \int_{D_{\epsilon}}f(u^{\epsilon})v\omega_{\epsilon}\dd \bm x-\int_D f(u)v\dd \bm x+\int_{D_{\epsilon}}gv\left|\nabla\omega_{\epsilon}\right|\dd \bm x-\int_{\partial D}gv\dd \sigma\\
			\lesssim& \|u^{\epsilon}-u\|_{H^1(D_{\epsilon};\omega_{\epsilon})}\|v\|_{L^2(D_{\epsilon};\omega_{\epsilon})}+\epsilon^{\frac 3 2}\left\|f(u)\right\|_{H^1(\Gamma_{\epsilon};\omega_{\epsilon})}\|v\|_{H^1(\Gamma_{\epsilon};\omega_{\epsilon})}\\
			&+\int_{D_{\epsilon}} gv\left|\nabla\omega_{\epsilon}\right|\dd \bm x-\int_{\partial D}gv\dd \sigma.
		\end{aligned}
	\end{equation}

	Combining \eqref{thm3-4}-\eqref{thm3-7} together, we get
	\begin{equation}
		\label{thm3-8}
		\begin{aligned}
			&\left(u_t^{\epsilon}-u_t,v\right)_{D_{\epsilon};\omega_{\epsilon}}+a^{\epsilon}(u^{\epsilon}-u,v)\\
			&\quad=\left(u_t,v\right)-\left(u_t,v\right)_{D_{\epsilon};\omega_{\epsilon}}+a(u,v)-a^{\epsilon}(u,v)+\ell^{\epsilon}(v)-\ell(v)\\	
			&\quad\lesssim\epsilon^{\frac 3 2}\|u_t\|_{H^1(\Gamma_{\epsilon};\omega_{\epsilon})}\|v\|_{H^1(\Gamma_{\epsilon};\omega_{\epsilon})}+\epsilon^{\frac 3 2}\left\|\divo\left(A\nabla u\right)\right\|_{H^1(\Gamma_{\epsilon};\omega_{\epsilon})}\|v\|_{H^1(\Gamma_{\epsilon};\omega_{\epsilon})}\\
			&\qquad-\int_{D_{\epsilon}} \bm n\cdot A\nabla u\left|\nabla\omega_{\epsilon}\right|v\dd\bm x+\left\|u^{\epsilon}-u\right\|_{H^1(D_{\epsilon};\omega_{\epsilon})}\|v\|_{L^2(D_{\epsilon};\omega_{\epsilon})}\\
			&\qquad+\epsilon^{\frac 3 2}\|f(u)\|_{H^1(\Gamma_{\epsilon};\omega_{\epsilon})}\|v\|_{H^1(\Gamma_{\epsilon};\omega_{\epsilon})}+\int_{D_{\epsilon}} gv|\nabla\omega_{\epsilon}|\dd \bm x.
		\end{aligned}
	\end{equation}
	Applying with Theorem \ref{thm2}, we know that
	\begin{equation}\label{thm3-9}
		\begin{aligned}
			\int_{D_{\epsilon}} gv|\nabla\omega_{\epsilon}|\dd \bm x-\int_{D_{\epsilon}}\bm n\cdot A\nabla uv|\nabla\omega_{\epsilon}|\dd \bm x
			&=\int_{D_{\epsilon}}\left(\bm n\cdot A\nabla u-g \right)v|\nabla \omega_{\epsilon}|\dd \bm x\notag\\
			&\lesssim\epsilon^{\frac 3 2}\|\bm n\cdot A\nabla u-g\|_{H^2(\Gamma_{\epsilon};\omega_{\epsilon})}\|v\|_{H^1(D_{\epsilon};\omega_{\epsilon})}.
		\end{aligned}
	\end{equation}
	Therefore, we have by  \eqref{thm3-7}   and \eqref{thm3-8} that
	\begin{equation}\label{thm3-10}
		\begin{aligned}
			&\left(u_t^{\epsilon}-u_t,v\right)_{D_{\epsilon};\omega_{\epsilon}}+a^{\epsilon}(u^{\epsilon}-u,v)\\
			&\quad\lesssim\epsilon^{\frac 3 2}\|u_t\|_{H^1(\Gamma_{\epsilon};\omega_{\epsilon})}\|v\|_{H^1(\Gamma_{\epsilon};\omega_{\epsilon})}+\epsilon^{\frac 3 2}\left\|\divo\left(A\nabla u\right)\right\|_{H^1(\Gamma_{\epsilon};\omega_{\epsilon})}\|v\|_{H^1(\Gamma_{\epsilon};\omega_{\epsilon})}\\
			&\qquad+\epsilon^{\frac 3 2}\|\bm n\cdot A\nabla u-g\|_{H^2(\Gamma_{\epsilon};\omega_{\epsilon})}\|v\|_{H^1(D_{\epsilon};\omega_{\epsilon})}+\epsilon^{\frac 3 2}\|f(u)\|_{H^1(\Gamma_{\epsilon};\omega_{\epsilon})}\|v\|_{H^1(D_{\epsilon};\omega_{\epsilon})}\\
			&\qquad +\left\|u^{\epsilon}-u\right\|_{H^1(D_{\epsilon};\omega_{\epsilon})}\|v\|_{L^2(D_{\epsilon};\omega_{\epsilon})}.
		\end{aligned}
	\end{equation}
	
	Since $|\Gamma_{\epsilon}|\lesssim \epsilon$, applying $v=u^{\epsilon}-u$ to \eqref{thm3-9} gives
	\begin{equation}
		\begin{aligned}
			\label{thm3-11}
			&\frac 1 2\frac{\dd}{\dd t}\|u^{\epsilon}-u\|_{L^2(D_{\epsilon};\omega_{\epsilon})}^2+\|\nabla(u^{\epsilon}-u)\|_{L^2(D_{\epsilon};\omega_{\epsilon})}^2\\
			&\quad\lesssim \epsilon^{\frac 3 2}\Big(\|u_t\|_{H^1(\Gamma_{\epsilon};\omega_{\epsilon})}+\|\divo(A\nabla u)\|_{H^1(\Gamma_{\epsilon};\omega_{\epsilon})}+\|\bm n\cdot A\nabla u-g\|_{H^2(\Gamma_{\epsilon};\omega_{\epsilon})}\\
			&\qquad+\|f(u)\|_{H^1(\Gamma_{\epsilon};\omega_{\epsilon})}+\left\|u^{\epsilon}-u\right\|_{L^2(D_{\epsilon};\omega_{\epsilon})}\Big)\|u^{\epsilon}-u\|_{H^1(\Gamma_{\epsilon};\omega_{\epsilon})}\\
			&\quad\lesssim  \epsilon^4 +\|u^{\epsilon}-u\|_{L^2(D_{\epsilon};\omega_{\epsilon})}^2+\|\nabla(u^{\epsilon}-u)\|_{L^2(D_{\epsilon};\omega_{\epsilon})}^2.
		\end{aligned}
	\end{equation}
	Thus
	\begin{equation*}
		\frac 1 2\frac{\dd}{\dd t}\|u^{\epsilon}-u\|_{L^2(D_{\epsilon};\omega_{\epsilon})}^2 \lesssim \epsilon^4+\|u^{\epsilon}-u\|_{L^2(D_{\epsilon};\omega_{\epsilon})}^2,
	\end{equation*}
	which implies 
	\begin{equation*}
		\|u^{\epsilon}-u\|_{L^2(D_{\epsilon};\omega_{\epsilon})}\lesssim \epsilon^2.
	\end{equation*}
	The proof is completed.
	
\end{proof}

\subsection{Optimal error estimate in the $H^1$ norm}

The derivation of following lemma has been presented in Lemma 4.2 in \cite{HaoJu2025}.

\begin{lemma}
	\label{lemma3}
	Suppose that $0<\epsilon\leq \epsilon_0$,  $g\in H^1(D_{\epsilon};\omega_{\epsilon}) $, and $f$ satisfies  Assumption \ref{assumption1}. Assume the exact solution $u\in H^2(D)$ for all $0\leq t\leq T$ and its extension  to $D_\epsilon$ fulfills Assumption  \ref{regularity_assump}. Then we have
	\begin{equation}
		\label{lemma3-1}
		\int_{\partial D}g(u^{\epsilon}-u)\dd \sigma+\int_{D_{\epsilon}} g\nabla\omega_{\epsilon}(u^{\epsilon}-u)\dd \bm x\leq C\epsilon^2+\frac 1 4\|\nabla u^{\epsilon}-\nabla u\|_{L^2(\Gamma_{\epsilon};\omega_{\epsilon})}^2,
	\end{equation}
	where $C$ is a constant independent of $\epsilon$.
\end{lemma}

\begin{theorem}[Error estimate in the $H^1$-norm]
	\label{thm5}
	Suppose that $0<\epsilon\leq \epsilon_0$,  $g\in L^{\infty}(D_{\epsilon})\cap H^2(D_{\epsilon};\omega_{\epsilon})$, $\kappa\leq A(\bm x)\leq \kappa^{-1}$ for all $\bm x \in D$ with some constant $\kappa>0$, and $f$ satisfies  Assumption \ref{mild_growth} and \ref{assumption1}. Assume the exact solution $u\in H^2(D_{\epsilon})$ for all $0\leq t \leq T$ and its extension to $D_\epsilon$ fulfills Assumption  \ref{regularity_assump}. Then we have
	\begin{equation}
		\label{thm5-1}
		\|u^{\epsilon}(t)-u(t)\|_{H^1(D_{\epsilon};\omega_{\epsilon})}\lesssim \epsilon, \quad \forall\ 0 \leq t\leq T,
	\end{equation}
	where the hidden constant is independent of $\epsilon$.
	
\end{theorem}

\begin{proof}
	First, we insert $v=2(u^{\epsilon}_t-u_t)$ into \eqref{eq4-1}, we have
	\begin{equation}
		\begin{aligned}
			\label{thm5-2}
			&2\left(u_t^{\epsilon}-u_t,u_t^{\epsilon}-u_t\right)_{D_{\epsilon},\omega_{\epsilon}}+2a^{\epsilon}(u^{\epsilon}-u,u_t^{\epsilon}-u_t)\\
			&\quad = \left(2\left(u_t,u_t^{\epsilon}-u_t\right)-2\left(u_t,u_t^{\epsilon}-u_t\right)_{D_{\epsilon};\omega_{\epsilon}}\right)+\left(2a(u,u_t^{\epsilon}-u_t)-2a^{\epsilon}(u,u_t^{\epsilon}-u_t)\right)\\
			&\quad \quad +\left(2\int_{D_{\epsilon}} f(u^{\epsilon})\left(u^{\epsilon}_t-u_t\right)\omega_{\epsilon}\dd\bm x-2\int_D f(u)\left(u_t^{\epsilon}-u_t\right)\dd\bm x\right)\\
			&\quad\quad+\left(2\int_{D_{\epsilon}}g(u_t^{\epsilon}-u_t)\left|\nabla\omega_{\epsilon}\right|\dd \bm x-2\int_{\partial D}g(u_t^{\epsilon}-u_t)\dd\sigma \right)=:\rm{I}_1+\rm{I}_2+\rm{I}_3+\rm{I}_4 .
		\end{aligned} 
	\end{equation}
	
	Since $u^{\epsilon}_t$ and $u_t$, as well as $u^{\epsilon}$ and $u$, have the same regularity, according to Lemma \ref{lemma1} and $L^2$-norm error estimate of diffuse domain method (see Theorem \ref{thm3}), we can easily obtain that
	\begin{equation}
		\begin{aligned}
			\label{thm5-3}
			\rm{I}_1 &\lesssim \epsilon^{\frac 1 2}\left\|u_t^{\epsilon}-u_t\right\|_{L^2(\Gamma_{\epsilon};\omega_{\epsilon})}\left\|u_t\right\|_{L^2(\Gamma_{\epsilon};\omega_{\epsilon})}\lesssim \epsilon^3.
		\end{aligned}
	\end{equation}
	
	Since the weighted function $\omega_{\epsilon}$ will vanish on the boundary of $D_{\epsilon}$, we can reformulate the term $\rm{I}_2$ with the variational formula, 
	\begin{equation*}
		\begin{aligned}
			\label{thm5-4}
			\rm{I}_2 =& 2\int_{\partial D} \bm{n}\cdot A\nabla u\left(u_t^{\epsilon}-u_t\right)\dd\sigma - 2\int_D\divo(A\nabla u)\left(u_t^{\epsilon}-u_t\right)\dd\bm x\\
			&-2\left(\int_{\partial D_{\epsilon}}\bm{n}\cdot A\nabla u\omega_{\epsilon}\left(u_t^{\epsilon}-u_t\right)\dd\sigma-\int_{D_{\epsilon}}\divo(A\nabla u\omega_{\epsilon})\left(u_t^{\epsilon}-u_t\right)\dd\bm x \right)\\
			=&\left(2\int_{D_{\epsilon}}\divo(A\nabla u)\left(u_t^{\epsilon}-u_t\right)\omega_{\epsilon}\dd\bm x-2\int_D\divo(A\nabla u)\left(u_t^{\epsilon}-u_t\right)\dd\bm x \right)\\
			&+\left(2\int_{\partial D} g\left(u_t^{\epsilon}-u_t\right)\dd\sigma+2\int_{D_{\epsilon}}A\nabla u\nabla\omega_{\epsilon}\left(u_t^{\epsilon}-u_t\right)\dd\bm x \right)=:\rm{II}_1+\rm{II}_2.
		\end{aligned}
	\end{equation*}
	Similar to the estimation of \eqref{thm5-3}, according to Lemma \ref{lemma1} and Theorem \ref{thm3}, we have
	\begin{equation}
		\begin{aligned}
			\label{thm5-5}
			\rm{II}_1 &\lesssim\epsilon^{\frac 1 2}\left\|u_t^{\epsilon}-u_t\right\|_{L^2(\Gamma_{\epsilon};\omega_{\epsilon})}\left\|\divo(A\nabla u_t)\right\|_{L^2(\Gamma_{\epsilon};\omega_{\epsilon})}\lesssim \epsilon^2.
		\end{aligned}
	\end{equation}
	Combining the term $\rm{II}_2$ and $\rm{I}_4$ together, by using the Cauchy-Schwarz inequality and the definition of $\omega_{\epsilon}$, we have 
	\begin{equation}
		\begin{aligned}
			\label{thm5-6}
			\rm{II}_2+\rm{I}_4 &= 2\int_{D_{\epsilon}}g\left(u_t^{\epsilon}-u_t\right)\left|\nabla \omega_{\epsilon}\right|\dd\bm x+2\int_{D_{\epsilon}}A\nabla u\nabla \omega_{\epsilon}\left(u_t^{\epsilon}-u_t\right)\dd\bm x\\
			&=2\int_{D_{\epsilon}} \left(g-\bm n \cdot A\nabla u\right)\left(u_t^{\epsilon}-u_t\right)\left|\nabla \omega_{\epsilon}\right|\dd\bm x\\
			&\lesssim \left(\int_{D_{\epsilon}} \left|\omega_{\epsilon}^{\frac 1 2}\left(u_t^{\epsilon}-u_t\right)\right|^2\dd\bm x\right)^{\frac 1 2}\left(\int_{D_{\epsilon}}\left|\left(g-\bm n\cdot A \nabla u\right)\left|\nabla \omega_{\epsilon}\right|/\omega_{\epsilon}^{\frac 1 2} \right|^2\dd\bm x \right)^{\frac 1 2}\\
			&\lesssim \sup\limits_{0\leq t \leq T} \left\|u_t^{\epsilon}-u_t\right\|_{L^2(D_{\epsilon};\omega_{\epsilon})}\left\|\bm n\cdot  A\nabla u-g\right\|_{L^{\infty}(D_{\epsilon})}\lesssim \epsilon^2.
		\end{aligned}
	\end{equation}
	
	As for the estimation of $\rm{I}_3$, due to the local continuity of $f$, we have
	\begin{equation}
		\begin{aligned}
			\label{thm5-7}
			\rm{I}_3=& \left(2\int_{D_{\epsilon}} f(u^{\epsilon})\left(u_t^{\epsilon}-u_t\right)\omega_{\epsilon}\dd\bm x-2\int_{D_{\epsilon}} f(u)\left(u_t^{\epsilon}-u_t\right)\omega_{\epsilon}\dd\bm x\right)\\
			&+\left(2\int_{D_{\epsilon}}f(u)\left(u_t^{\epsilon}-u_t \right)\omega_{\epsilon}\dd\bm x-2\int_D f(u)\left(u_t^{\epsilon}-u_t \right)\dd\bm x\right)\\
			\lesssim & \|f(u^{\epsilon})-f(u)\|_{L^2(D_{\epsilon};\omega_{\epsilon})}\|u_t^{\epsilon}-u_t\|_{L^2(D_{\epsilon};\omega_{\epsilon})}\\
			&+\epsilon^{\frac 1 2}\|f(u)\|_{L^2(\Gamma_{\epsilon};\omega_{\epsilon})}\|u_t^{\epsilon}-u_t\|_{L^2(\Gamma_{\epsilon};\omega_{\epsilon})}\\
			\lesssim & \|u^{\epsilon}-u\|_{H^1(D_{\epsilon};\omega_{\epsilon})}\|u_t^{\epsilon}-u_t\|_{L^2(D_{\epsilon};\omega_{\epsilon})}+\epsilon^{\frac 1 2}\|f(u)\|_{L^2(\Gamma_{\epsilon};\omega_{\epsilon})}\|u_t^{\epsilon}-u_t\|_{L^2(\Gamma_{\epsilon};\omega_{\epsilon})}\\
			\lesssim & \epsilon^2 + \|\nabla u^{\epsilon}-\nabla u\|_{L^2(D_{\epsilon};\omega_{\epsilon})}^2,
		\end{aligned}
	\end{equation}
	where the last three inequality uses the conclusion of Lemma \ref{lemma1} and Theorem \ref{thm3}. Inserting \eqref{thm5-3}-\eqref{thm5-7} to \eqref{thm5-2}, we can easily derive the conclusion.

\end{proof}

\section{Numerical experiments}\label{numerical}

In this section, we will present some numerical experiments to verify the error estimates (Theorems \ref{thm3}  and Theorem \ref{thm5} ) obtained in Section \ref{theory} and demonstrate the performance of the DDM. We apply the finite element method with the bilinear basis functions  for space discretization and  the BDF2 scheme for time stepping to solve the 
DDM-transformed problem \eqref{weighted_variational} defined on a larger rectangular domain, which is of second-order  in time, and first-order  in space with respect to the $H^1$ norm and second-order  in space with respect to the $L^2$ norm.
To demonstrate  the convergence order of the diffuse domain solution with respect to the interface width parameter $\epsilon$, i.e., $\left\|u(t)-u^{\epsilon}(t)\right\|_{L^2(D_{\epsilon};\omega_{\epsilon})}$ and $\left\|u(t)-u^{\epsilon}(t)\right\|_{H^1(D_{\epsilon};\omega_{\epsilon})}$, we will
take  $\epsilon$ to be much larger than the spatial mesh size and the time step size (so that the solution errors are mainly caused by the DDM approximation). All tests are done using Matlab on a laptop with Intel Ultra 9 185H, 2.30GHz CPU and 32GB memory.

\subsection{The case of constant diffusion coefficient}


\begin{example}
	\label{ex1}
	In this example, we consider the following two-dimensional diffusion problem with a constant diffusion coefficient: for $0 \leq t \leq T$, 
	\begin{equation}
		\left\{\begin{split}
			&u_t=3\Delta u +f(t, x, y, u), \ &&(x,y)\in D\\
			&u(0,x,y) = \left(\frac 5 2 x^2-5x\right)\left(\frac 5 2 y^2-5y \right), \ \  &&(x, y) \in D,
		\end{split}
		\right.
	\end{equation}
	where
	\begin{equation*}
		\begin{split}
			f(t, x, y, u) =& u-e^{-\pi^2 t}\Bigg((\pi^2+1)\left(\frac 5 2x^2-5x\right)\left(\frac 5 2 y^2-5y\right)\\
			&+15\left(\frac 5 2 x^2-5x\right)+15\left(\frac 5 2y^2-5y\right) \Bigg).
		\end{split}
	\end{equation*}
	The exact solution is given by $$u(t,x,y)=e^{-\pi^2 t}\left(\frac 5 2x^2-5x \right)\left(\frac 5 2y^2-5y \right).$$
	Two different domains $D$ are considered: one is a circular domain defined by
	\begin{equation}
		\label{circle}
		D=\left\{(x,y)\;\big|\; x^2+y^2=\frac{1}{16} \right\},
	\end{equation}
	and the other is a flower-shaped domain defined by
	\begin{equation}
		\label{flower}
		D=\left\{(x,y)\;\big|\; x^2+y^2-\left(0.18-0.03\sin\left(4\arctan\left(\frac y x\right)\right)\right)^2=0 \right\}.
	\end{equation}
	The Neumann boundary condition is imposed correspondingly, and 
	the terminal time is chosen as $T=0.5$. 
\end{example}

We run the DDM  on extended rectangular domain $\Omega=[-1/2,1/2]\times[-1/2, 1/2]$ ($D\subset\Omega$). We take the time steps $N_T=512$ (i.e., $\Delta \tau = T/N_T=1/1024$ and the uniformly spatial mesh with  $h_x = h_y=1/512$. 
To ensure the approximation error will dominate the numerical error, we set the interface thickness $\epsilon$=1/8, 1/16, 1/32, 1/64, respectively. Table \ref{tab 1} shows all the numerical errors measured in the weighted $L^2$ and $H^1$ norms and their corresponding convergence rates. Figure \ref{fig 1} also presents the simulated phase structures of the numerical solutions at the terminal time with the interface thickness $\epsilon$=1/16, 1/32, 1/64, respectively. From the Table \ref{tab 1}, we can observe that the diffuse domain method will converge at second order in the weighted $L^2$ norm and first order in the weighted $H^1$ norm, which coincide very well with the error estimates derived in Theorems \ref{thm3} and \ref{thm5}. For the Figure \ref{fig 1}, as the interface thickness decreasing, the transition zone in the figure becomes narrower and narrower. Also, the contour of the approximated region matches the exact region better and better. The inner phase structure depends on the intensity of fluctuations in the exact solution $u$ and the value of weighted function $\omega_{\epsilon}$. Numerical errors tend to be more significant near the domain boundaries, around singularity points, and in areas where the exact solution changes rapidly, whereas in the interior of the domain and in regions where the exact solution varies smoothly, numerical errors tend to be smaller.


\begin{table}[!htbp]
	\centering
	\caption{Numerical results on the solution errors measured in the weighted $L^2$ and $H^1$ norms and corresponding convergence rates at the terminal time $T=0.5$ produced by the DDM in Example \ref{ex1}.}
	\begin{tabular}{|ccccc|}
		\hline
		$\epsilon$ & $\|u^{\epsilon}-u(t_n)\|_{L^2(D_{\epsilon};\omega_{\epsilon})}$ & CR & $\|u^{\epsilon}-u(t_n)\|_{H^1(D_{\epsilon};\omega_{\epsilon})}$ & CR\\
		\hline
		\multicolumn{5}{|c|}{Convergence tests for the circular domain}\\
		\hline
		1/8 &  1.0300e-02 & - & 1.0600e-02 & - \\
		1/16 & 2.6000e-03 & 1.99 & 2.7000e-03 & 1.97 \\
		1/32 & 6.3631e-04 & 2.03 & 7.2971e-04 & 1.89\\
		1/64 & 1.5420e-04 & 2.04 & 3.9543e-04 & 0.88 \\
		\hline
		\multicolumn{5}{|c|}{Convergence tests for the flower-shaped domain}\\
		\hline
		1/8 & 8.5000e-03 & - & 8.7000e-03 & - \\
		1/16 & 2.1000e-03 & 2.02 & 2.2000e-03 & 1.98 \\
		1/32 & 5.2652e-04 & 2.00 & 5.7516e-04 & 1.94 \\
		1/64 & 1.3043e-04 & 2.01 & 2.5432e-04 & 1.18 \\
		\hline
	\end{tabular}
	\label{tab 1}
\end{table}

\begin{figure}[htbp]
	\centering
	\subfigure{
		\label{fig1-1}
		\centering
		\includegraphics[width = 120pt,height=110pt]{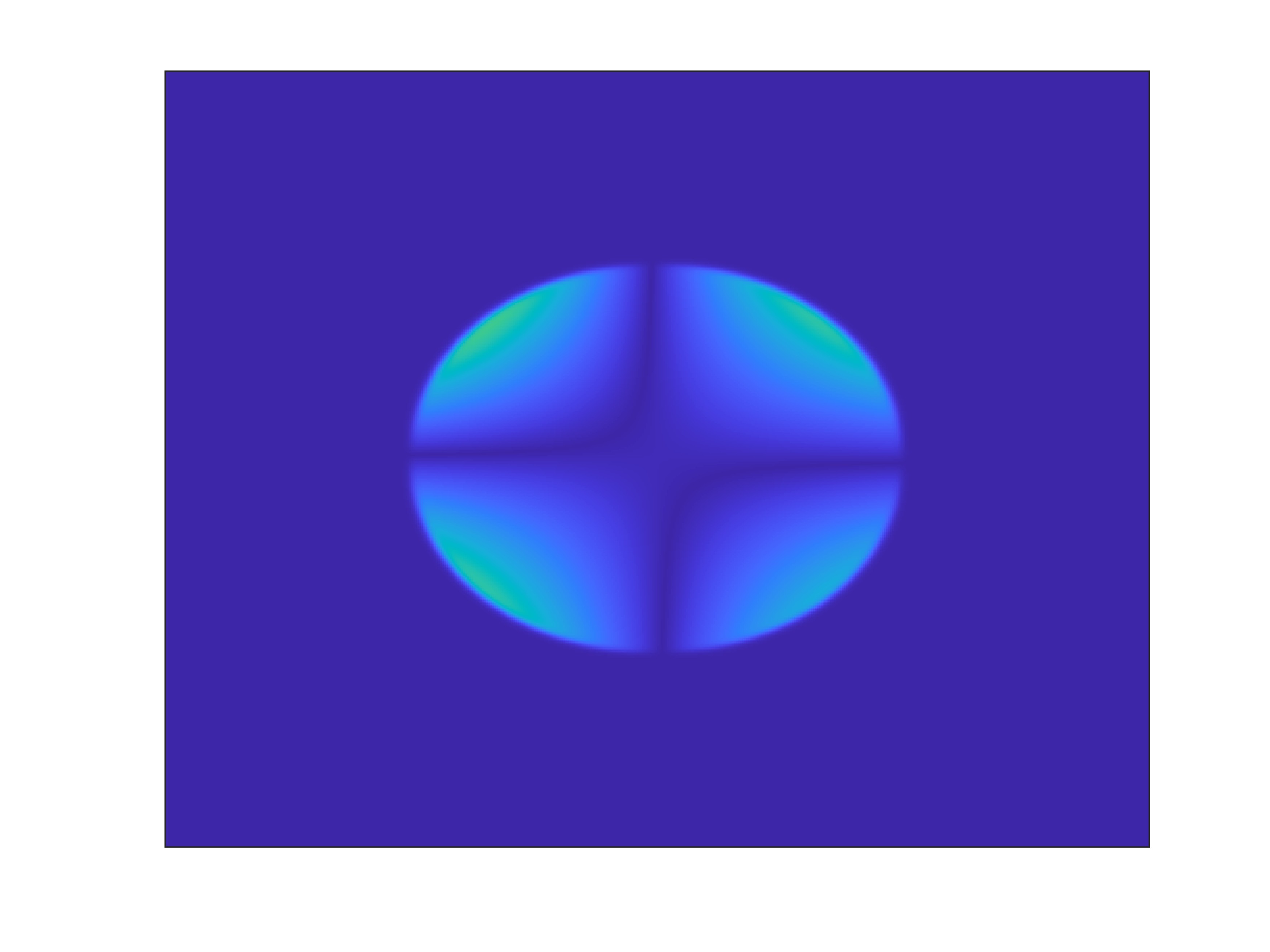}
	}\hspace{-0.8cm}
	\subfigure{
		\label{fig1-2}
		\centering
		\includegraphics[width = 120pt,height=110pt]{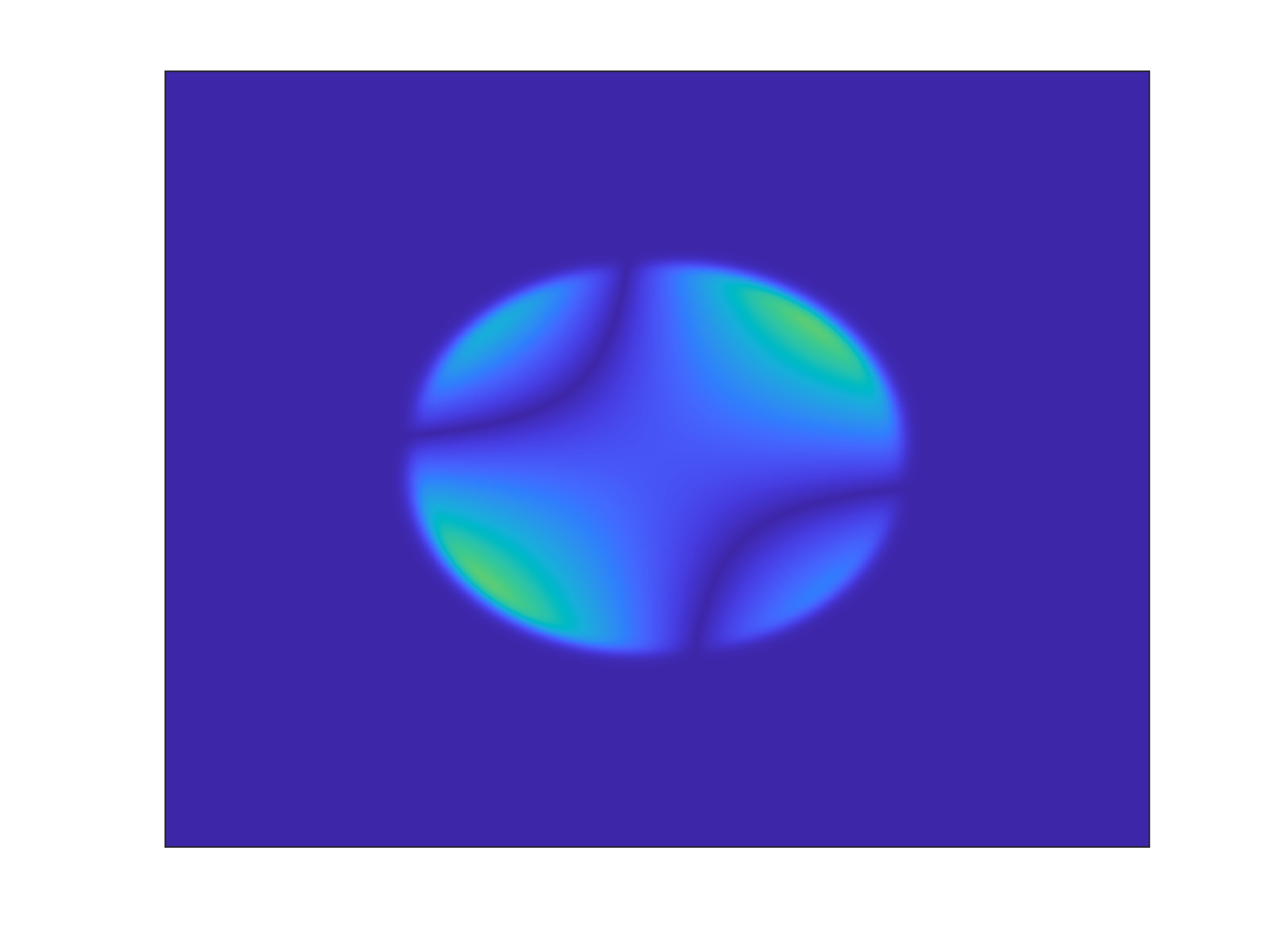}
	}\hspace{-0.8cm}
	\subfigure{
		\label{fig1-3}
		\centering
		\includegraphics[width = 130pt,height=110pt]{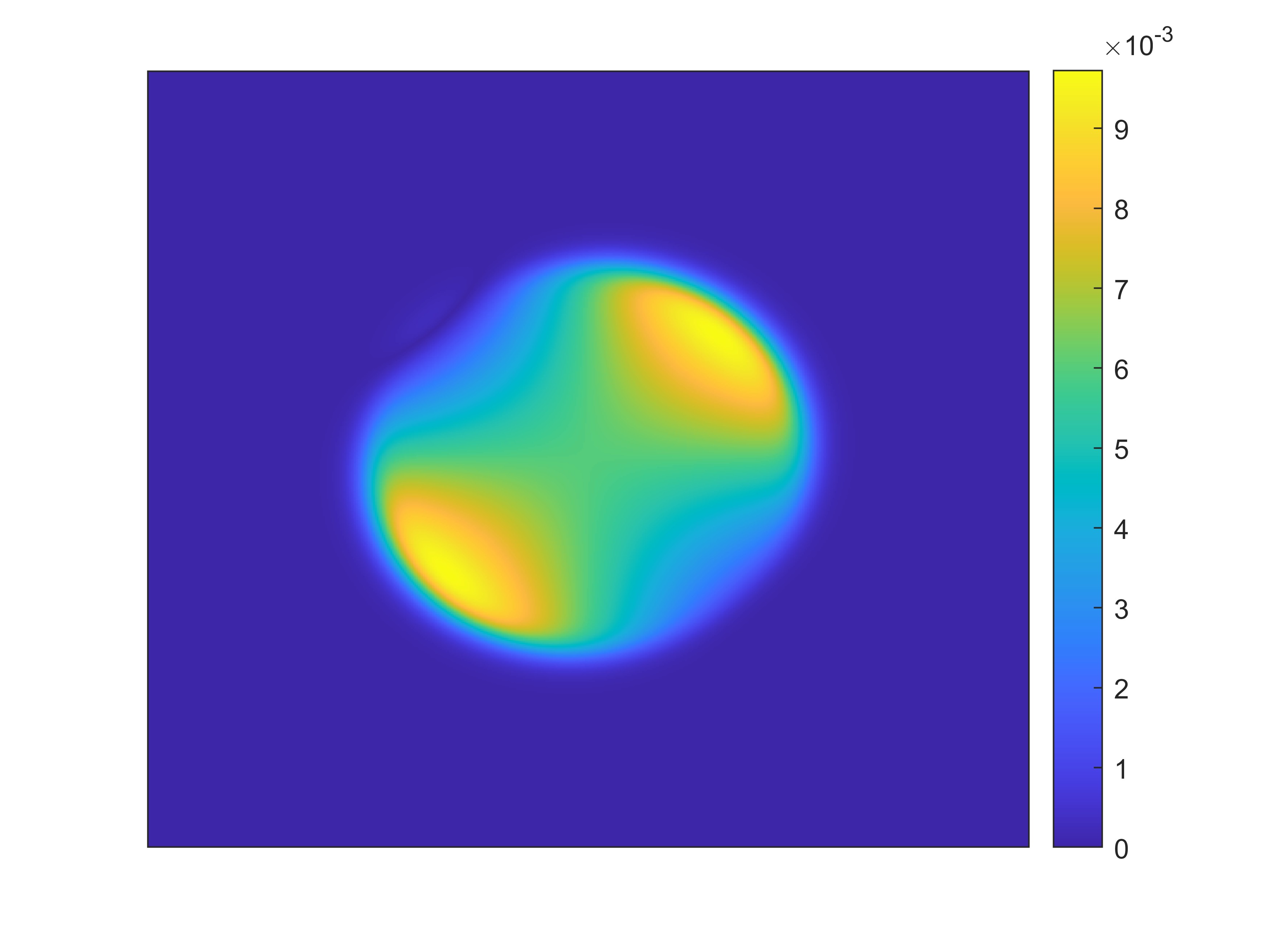}
	}
	
	\subfigure{
		\label{fig4-1}
		\centering
		\includegraphics[width = 120pt,height=110pt]{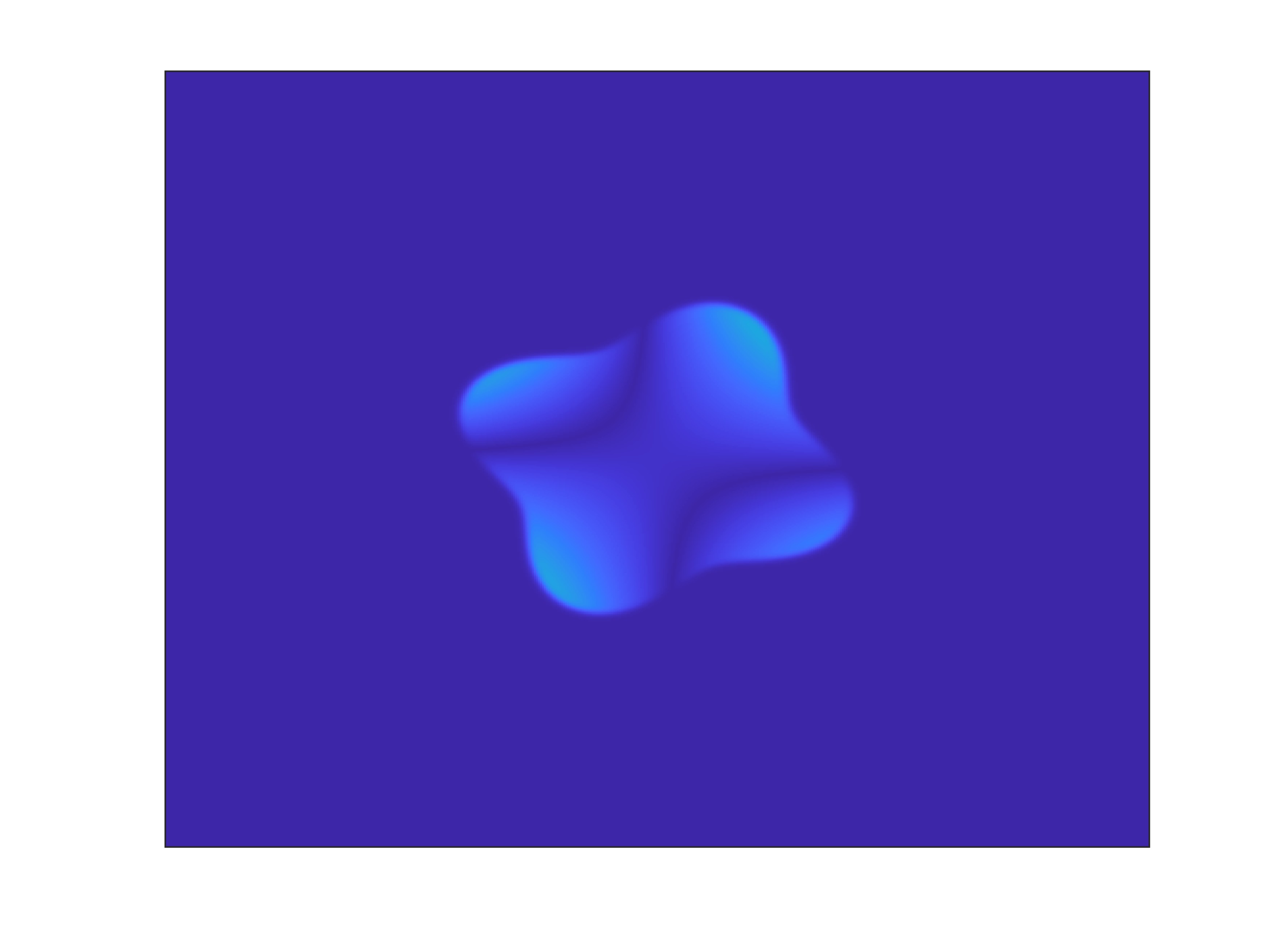}
	}\hspace{-0.8cm}
	\subfigure{
		\label{fig4-2}
		\centering
		\includegraphics[width = 120pt,height=110pt]{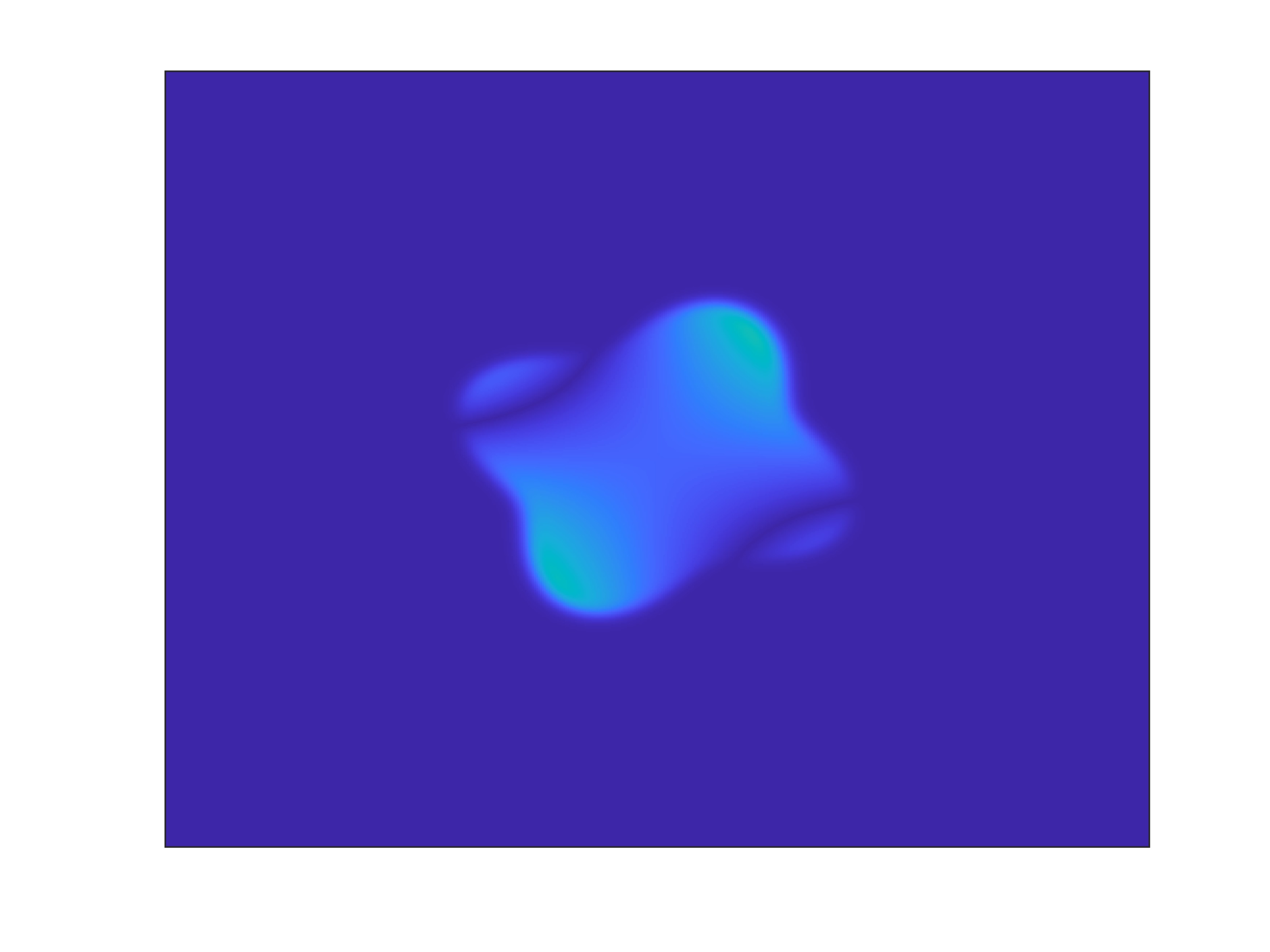}
	}\hspace{-0.8cm}
	\subfigure{
		\label{fig4-3}
		\centering
		\includegraphics[width = 130pt,height=110pt]{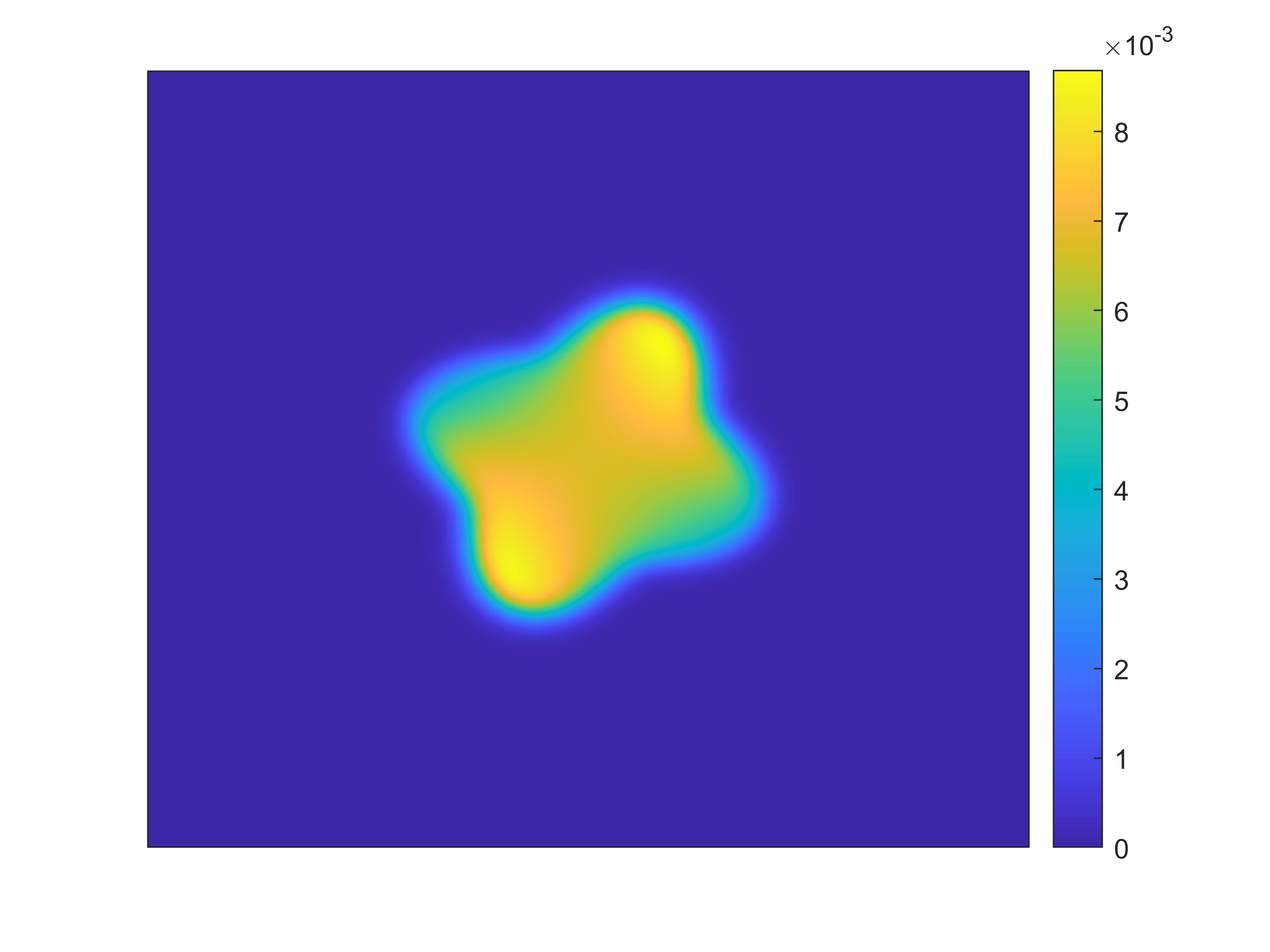}
	}
	\vspace{-0.5cm}
	\caption{The phase structures of the numerical solutions at the terminal time $T=0.5$ produced by the DDM approach with interface thickness $\epsilon=1/64,1/32,1/16$ (from left to right) for Example \ref{ex1}  in the circular domain (top row) and the flower-shaped domain (bottom row).}
	\label{fig 1}
\end{figure}

\subsection{The case of varying diffusion coefficient}


\begin{example}
	\label{ex2}
	In this example, we consider the following two-dimensional  diffusion problem  in the same circular domain and the flower-shaped domain used in Example \ref{ex1},  but with varying diffusion coefficient: for $0 \leq t \leq T$, 
	\begin{equation*}
		\left\{\begin{split}
			&u_t=\nabla \cdot \left((x^2+y^2+4)\nabla u\right)+f(t,x,y), \ &&(x,y)\in D\\
			&u(0,x,y) = \left(2 x^2+4x\right)\left(2 y^2-4y \right), \ \  &&(x, y) \in D,
		\end{split}
		\right.
	\end{equation*}
	where 
	\begin{equation*}
		\begin{split}
			f(t,x,y) =& -e^{-\pi^2 t}\Big(\pi^2\left(2 x^2-4x\right)\left(2 y^2-4y\right)+\left(4(x^2+y^2+4)+2x(4x-4)\right)\\
			&\cdot\left(2 y^2-4y \right)+\left(4(x^2+y^2+4)+2y(4y-4)\right)\left(2 x^2-4x \right)\Big),
		\end{split}
	\end{equation*}
	In this case, the exact solution is given by $$u(t,x,y)=e^{-\pi^2 t}\left(2 x^2-4x \right)\left(2 y^2-4y \right).$$
	Circular domain and flower-shaped domain defined in \eqref{circle} and \eqref{flower} are considered again. The Neumann boundary condition is correspondingly imposed and the terminal time is  set to $T=0.5$.
\end{example}

We verify the approximation accuracy of DDM by fixing $N_T=512$ (i.e., $\Delta \tau = T/N_T=1/1024$) and the uniform spatial meshes with mesh size $h_x=h_y=1/512$. The extended rectangular domain is set to be $\Omega=[-1/2,1/2]\times[-1/2,1/2]$ ($D \subset \Omega$). Then we set the interface thickness $\epsilon=1/8,1/16,1/32, 1/64$ respectively for approximation accuracy tests, ensuring the spatial mesh sizes and temporal step size are much finer than the interface thickness $\epsilon$. Table \ref{tab 2} shows the solution errors measured in the weighted $L^2$ and $H^1$ norms, including the corresponding convergence rates. From the Table \ref{tab 2}, we can conclude that the convergence rate for DDM method is two for the weighted $L^2$ norm. Though the convergence rate is slightly higher than one for the weighted $H^1$ norm, the numerical results roughly match our theoretical conclusion. Fixing the same spatial meshes and temporal partitions, the simulated phase structures of the numerical results at terminal time are shown in the Fig \ref{fig 2} with interface thickness $\epsilon$=1/16, 1/32, 1/64, respectively. It's obvious that with the interface thickness decreasing, the transition zone gradually becomes narrower and narrower, and the shape of the approximated region resembles the original irregular domain better and better. Similar as the Example \ref{ex1}, the phase structure of the numerical solution mainly depends on the intensity of fluctuations in the exact solution $u$ and the weighted function $\omega_{\epsilon}$. The numerical error tends to be more significant near the domain boundaries, around singularity points and in areas where the objective function changes rapidly, while that tends to be smaller in the interior of the domain due to the smoothness of the function.

\begin{table}[htbp]
	\centering
	\caption{Numerical results on the solution errors measured in the weighted $L^2$ and $H^1$ norms and corresponding convergence rates at the terminal time $T=0.5$ produced by the DDM in Example \ref{ex2}.}
	\begin{tabular}{|ccccc|}
		\hline
		$\epsilon$ & $\|u^{\epsilon}-u(t_n)\|_{L^2(D;\omega_{\epsilon})}$ & CR & $\|u^{\epsilon}-u(t_n)\|_{H^1(D;\omega_{\epsilon})}$ & CR\\
		\hline
		\multicolumn{5}{|c|}{Approximation tests for circle domain}\\
		\hline
		1/8 &  9.2000e-03 & - & 9.3000e-03 & - \\
		1/16 & 2.3000e-03  & 2.00 & 2.4000e-03 & 1.95 \\
		1/32 & 5.6595e-04  & 2.02 & 6.6273e-04 & 1.86 \\
		1/64 & 1.3832e-04  & 2.03 & 2.8219e-04 & 1.23\\
		\hline
		\multicolumn{5}{|c|}{Approximation tests for flower-shaped domain}\\
		\hline
		1/8 & 7.4000e-03 & - & 7.5000e-03 & - \\
		1/16 & 1.8000e-03 & 2.04 & 1.9000e-03 & 1.98 \\
		1/32 & 4.5962e-04 & 1.97 & 5.3160e-04 & 1.84 \\
		1/64 & 1.1410e-04 & 2.01 & 2.1916e-04 & 1.28 \\
		\hline
	\end{tabular}
	\label{tab 2}
\end{table}

\begin{figure}[htbp]
	\centering
	\subfigure{
		\label{fig2-1}
		\centering
		\includegraphics[width = 120pt,height=110pt]{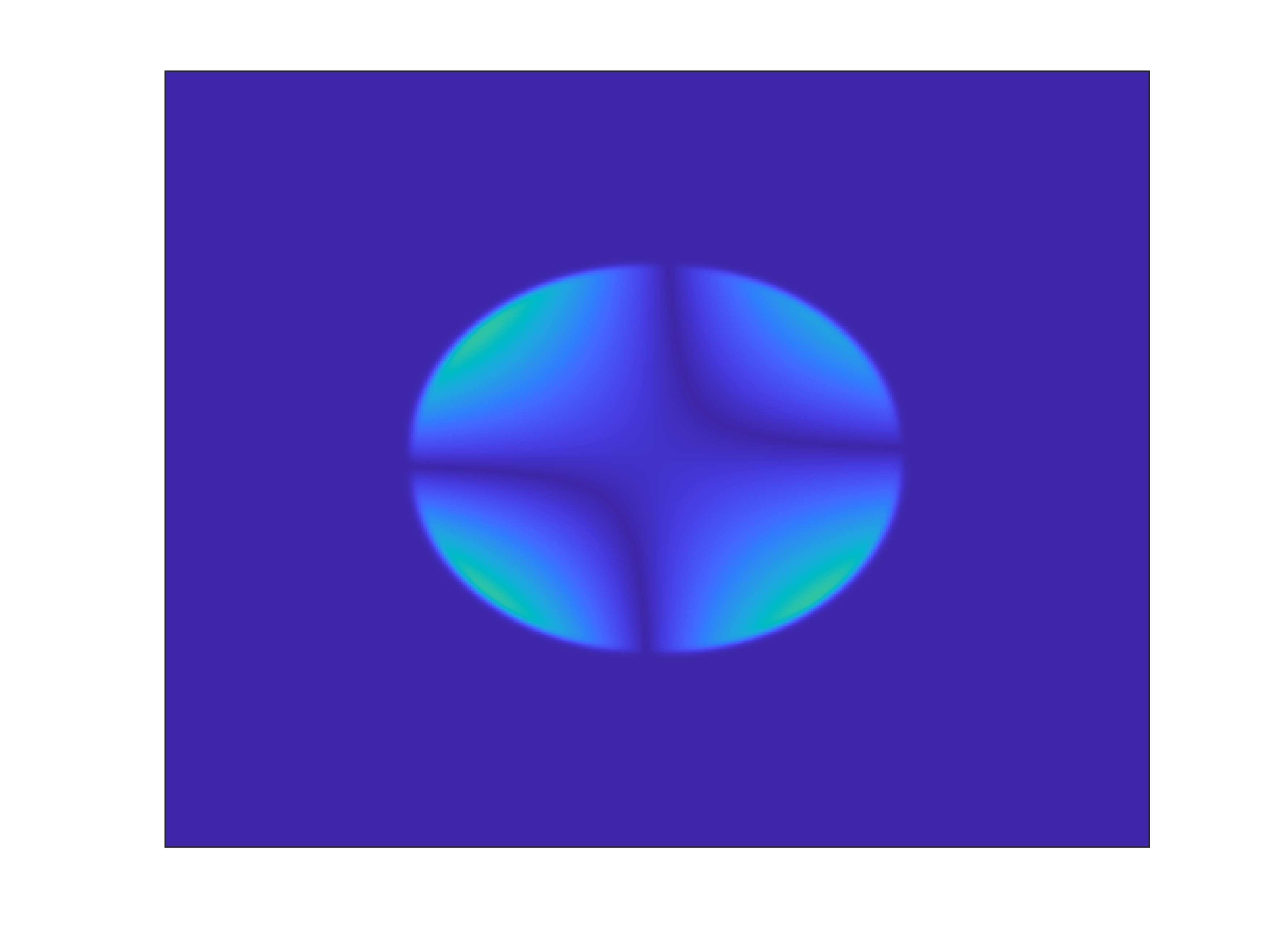}
	}\hspace{-0.8cm}
	\subfigure{
		\label{fig2-2}
		\centering
		\includegraphics[width = 120pt,height=110pt]{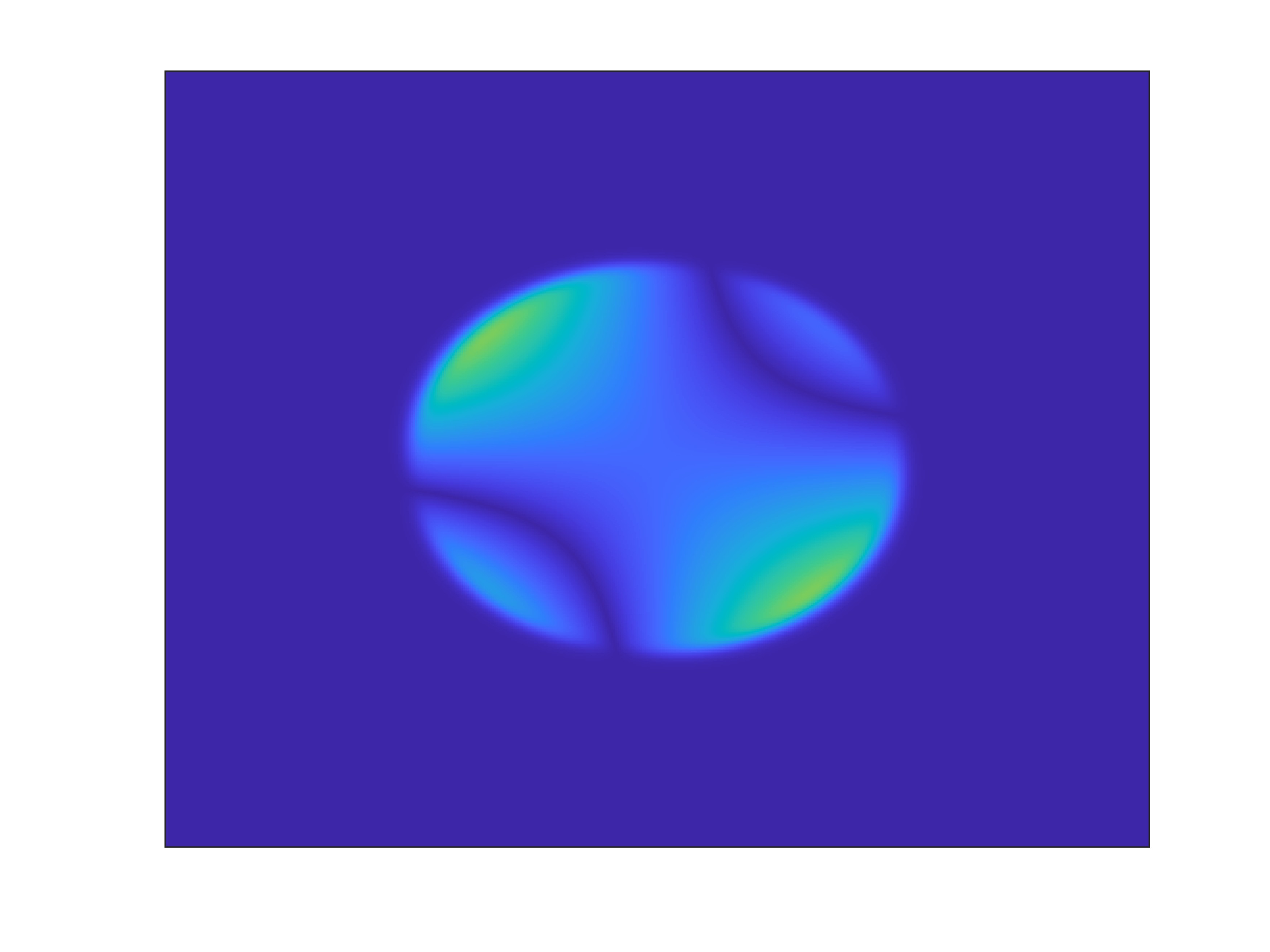}
	}\hspace{-0.8cm}
	\subfigure{
		\label{fig2-3}
		\centering
		\includegraphics[width = 130pt,height=110pt]{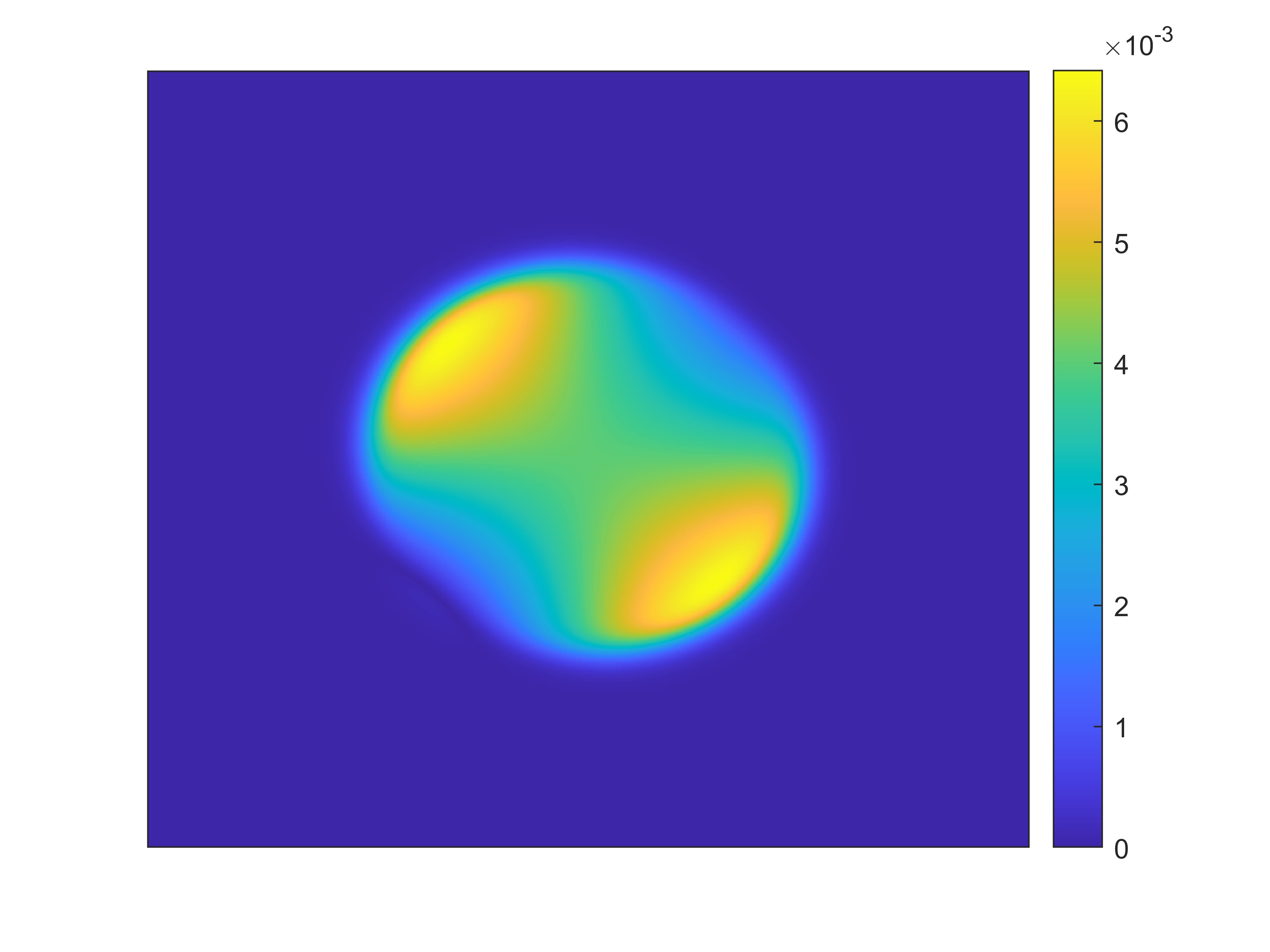}
	}
	
	\subfigure{
		\label{fig3-1}
		\centering
		\includegraphics[width = 120pt,height=110pt]{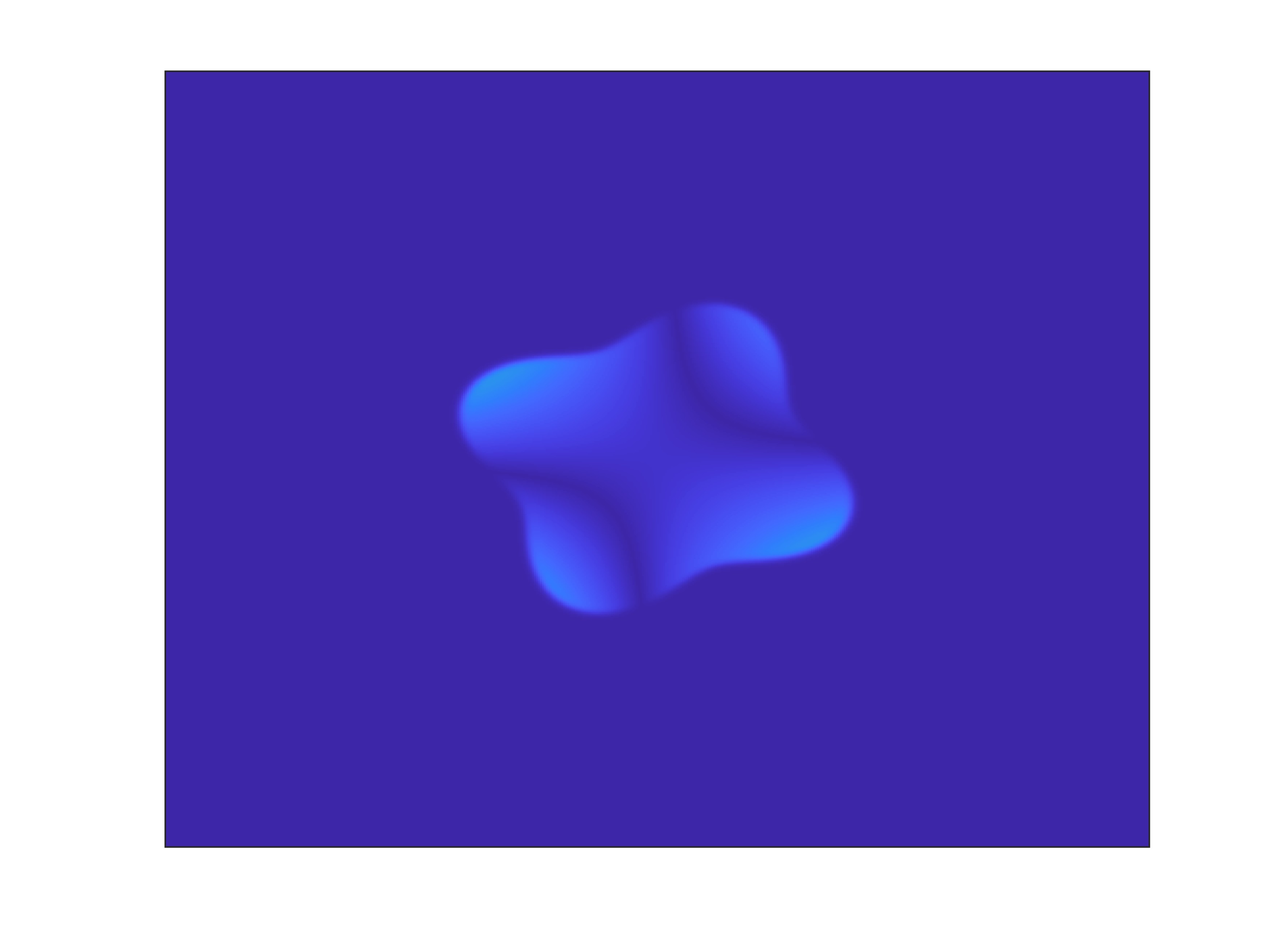}
	}\hspace{-0.8cm}
	\subfigure{
		\label{fig3-2}
		\centering
		\includegraphics[width = 120pt,height=110pt]{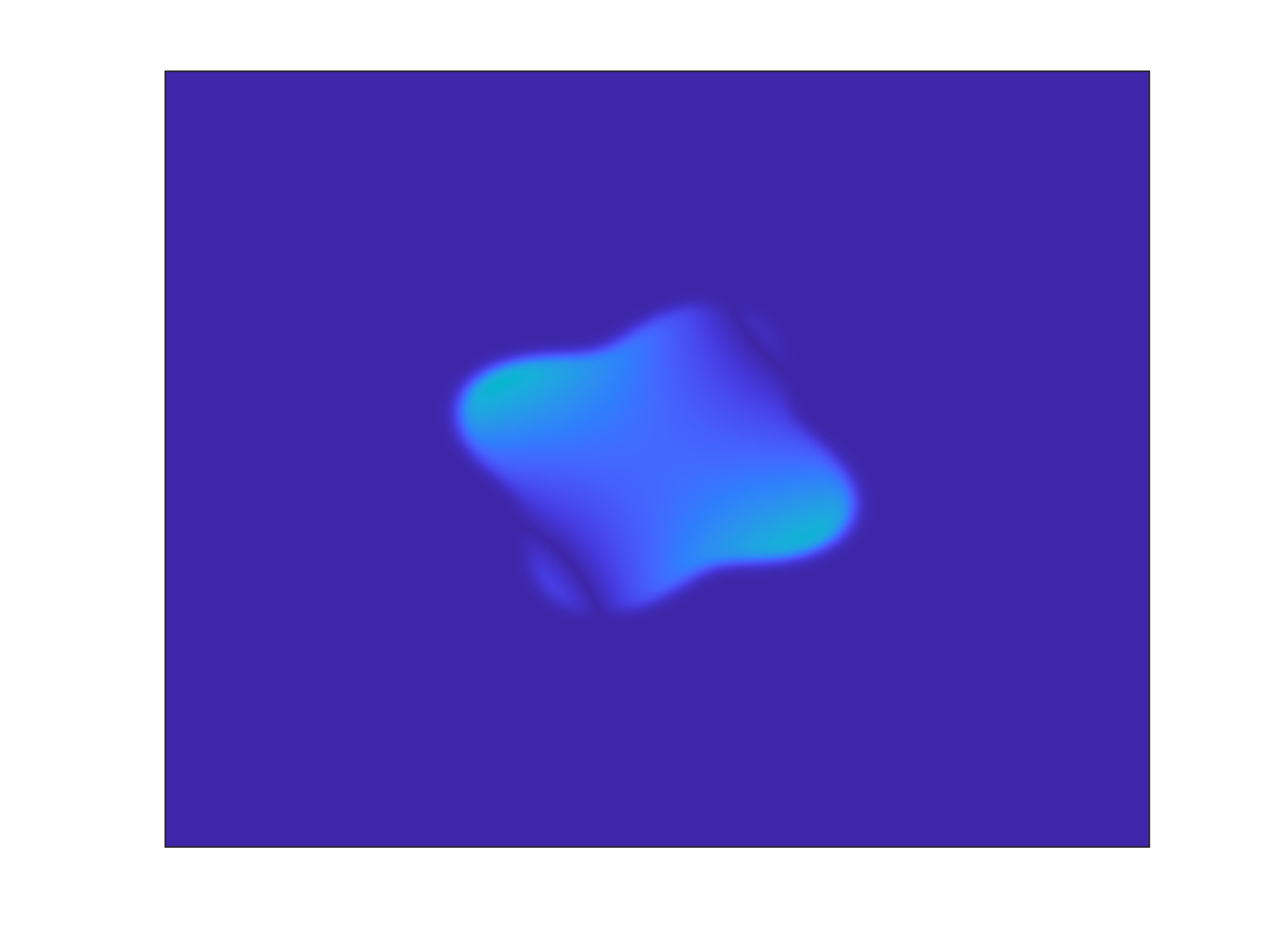}
	}\hspace{-0.8cm}
	\subfigure{
		\label{fig3-3}
		\centering
		\includegraphics[width = 130pt,height=110pt]{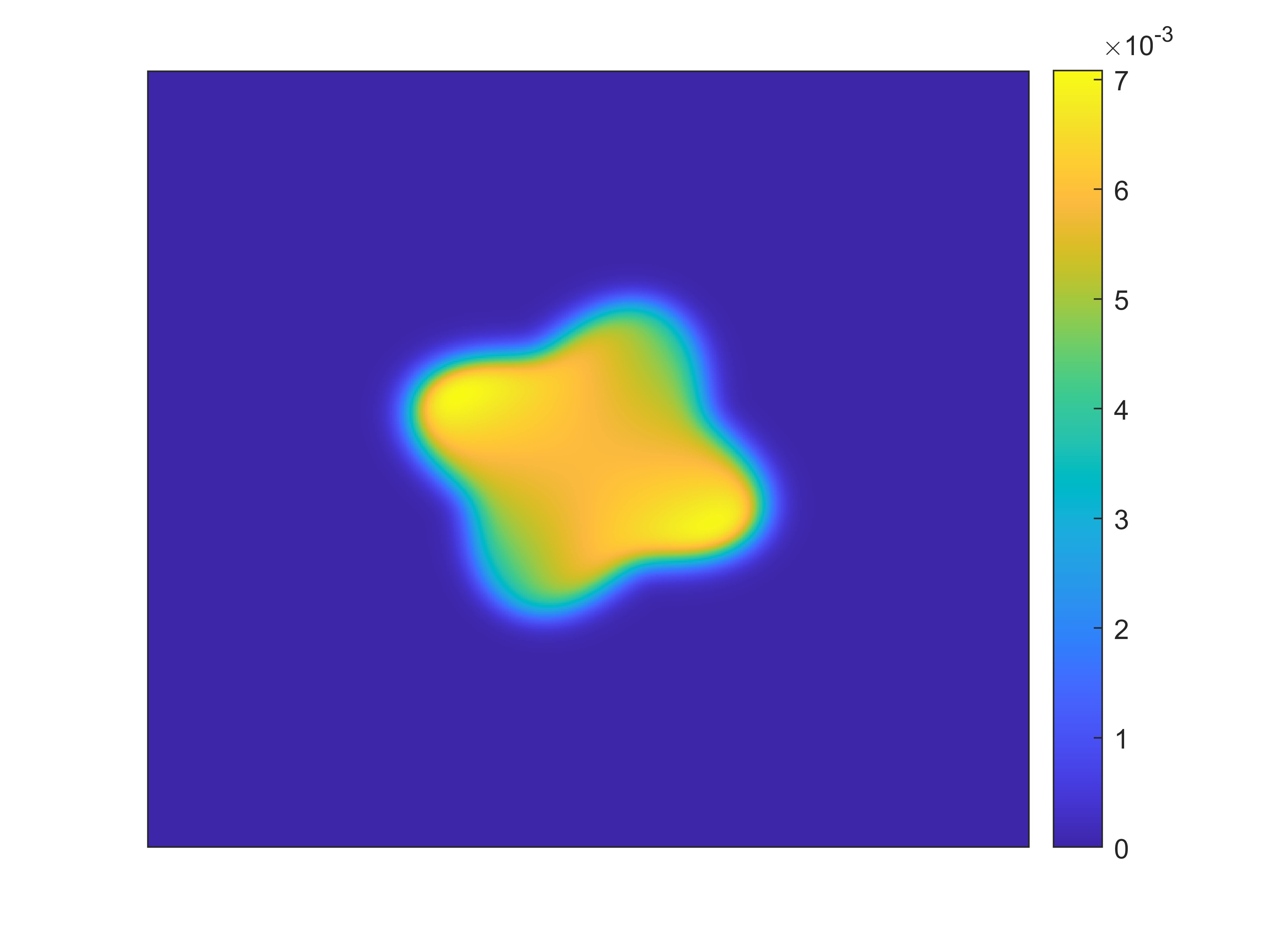}
	}
	\vspace{-0.5cm}
	\caption{The phase structures of the numerical results at the terminal time $T=0.5$ produced by the DDM approach with interface thickness $\epsilon=1/64,1/32,1/16$ (from left to right) for Example \ref{ex2} in the circular domain (top row) and the flower-shaped domain (bottom row).}
	\label{fig 2}
\end{figure}

\subsection{Allen-Cahn equation with double-well potential function}
\begin{example}
	\label{ex3}
	In this example, we consider the traveling wave problem governed by the following 2D Allen-Cahn equation with double-well potential function in the same flower-shaped domain used in Example \ref{ex1}.
	\begin{equation*}
		\label{Allen-Cahn}
		\left\{
		\begin{split}
			&u_t = \Delta u - \frac{1}{\widetilde{\epsilon}^2}(u^3-u), \quad (x,y)\in \Omega,\ 0\leq t\leq T,\\
			&u(0,x,y) = \frac 1 2 \left(1-\tanh\left(\frac{x}{2\sqrt{2}\widetilde{\epsilon}}\right)\right), \quad (x,y)\in \Omega,
		\end{split}
		\right.
	\end{equation*}
	where $\Omega=\left(-\frac1 2, \frac 1 2\right)\times\left(-\frac 1 2,\frac 1 2\right)$. The exact solution is given by
	\begin{equation*}
		u(t,x,y)=\frac{1}{2}\left(1-\tanh\left(\frac{x-st}{2\sqrt{2}\widetilde{\epsilon}}\right) \right),
	\end{equation*}
	where $s=\frac{3}{\sqrt{2}\widetilde{\epsilon}}$, and the Neumann boundary condition is correspondingly imposed, which is clearly nonhomogeneous. The terminal time is taken to be $T=1.85\widetilde{\epsilon}$.
\end{example}

We set the interface thickness parameter of Allen-Cahn equation is $\widetilde{\epsilon}=0.01$ (i.e., $T=0.185$). For the approximation accuracy test, we run the DDM scheme with fixed temporal meshes $N_T=1024$, uniformly spatial meshes $N_x\times N_y=256\times 64$ and refined interface thickness $\epsilon=$1/4, 1/8, 1/16, 1/32, respectively, so compared with the temporal error and spatial error, the approximation error will dominate the numerical error. All numerical results are shown in Table \ref{tab 3}, including the solution errors measured in the weighted $L^2$ and $H^1$ norms and corresponding convergence rates. For brevity, in this case, we only tested the problem on the flower-shaped domain. We still observe the roughly second-order convergence in the weighted $L^2$-norm and the convergence rate drops to 1 in the weighted $H^1$-norm, which basically match the theoretical results. Fixing the same spatial meshes and temporal partitions, the simulated phase structures of the numerical solution at terminal time are presented in the Fig \ref{fig 5} with interface thickness $\epsilon$=1/8, 1/16, 1/32, respectively. It's obvious that with the interface thickness decreasing, the transition zone gradually becomes narrower and narrower, and the shape of the approximated region fits the original region better and better. For all the values of $\epsilon$, the numerical solution will converge to one, which matches the theoretical results of Allen-Cahn equation.
\begin{table}[htbp]
	\centering
	\caption{Numerical results on the solution errors measured in the weighted $L^2$ and $H^1$ norms and corresponding convergence rates at the terminal time $T$ produced by the DDM in Example \ref{ex3}.}
	\begin{tabular}{|ccccc|}
		\hline
		$\epsilon$ & $\|u^{\epsilon}-u(t_n)\|_{L^2(D;\omega_{\epsilon})}$ & CR & $\|u^{\epsilon}-u(t_n)\|_{H^1(D;\omega_{\epsilon})}$ & CR\\
		\hline
		1/4 & 2.2000e-03 & - & 6.3700e-02 & -\\
		1/8 & 2.4128e-04 & 3.19 & 1.6200e-02 & 1.98 \\
		1/16 & 6.8178e-05 & 1.82 & 5.4000e-03 & 1.58 \\
		1/32 & 2.2160e-05 & 1.62 & 2.7000e-03 & 1.00 \\
		\hline
	\end{tabular}
	\label{tab 3}
\end{table}

\begin{figure}[htbp]
	\centering
	\subfigure{
		\label{fig5-1}
		\centering
		\includegraphics[width = 120pt,height=110pt]{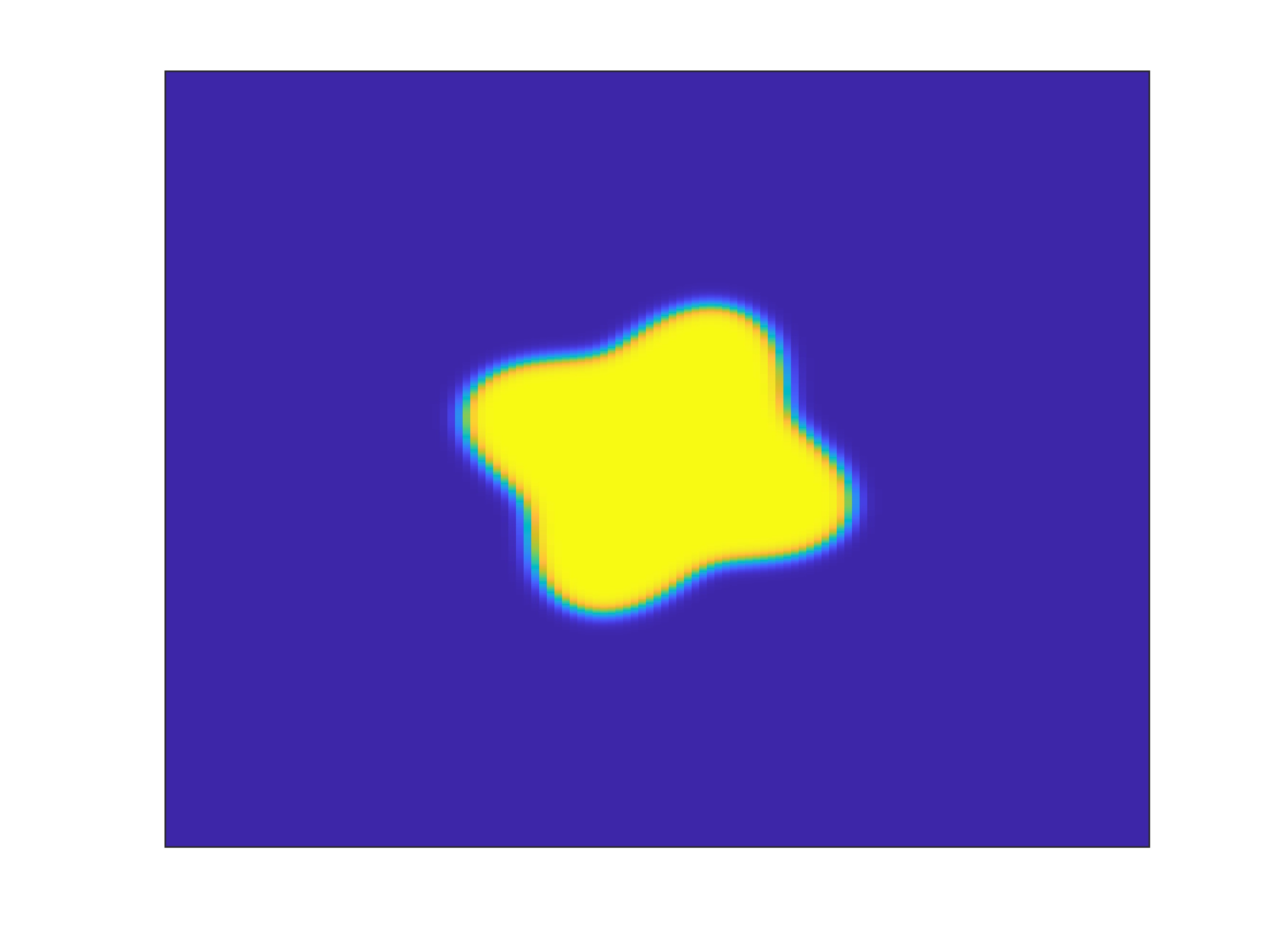}
	}\hspace{-0.8cm}
	\subfigure{
		\label{fig5-2}
		\centering
		\includegraphics[width = 120pt,height=110pt]{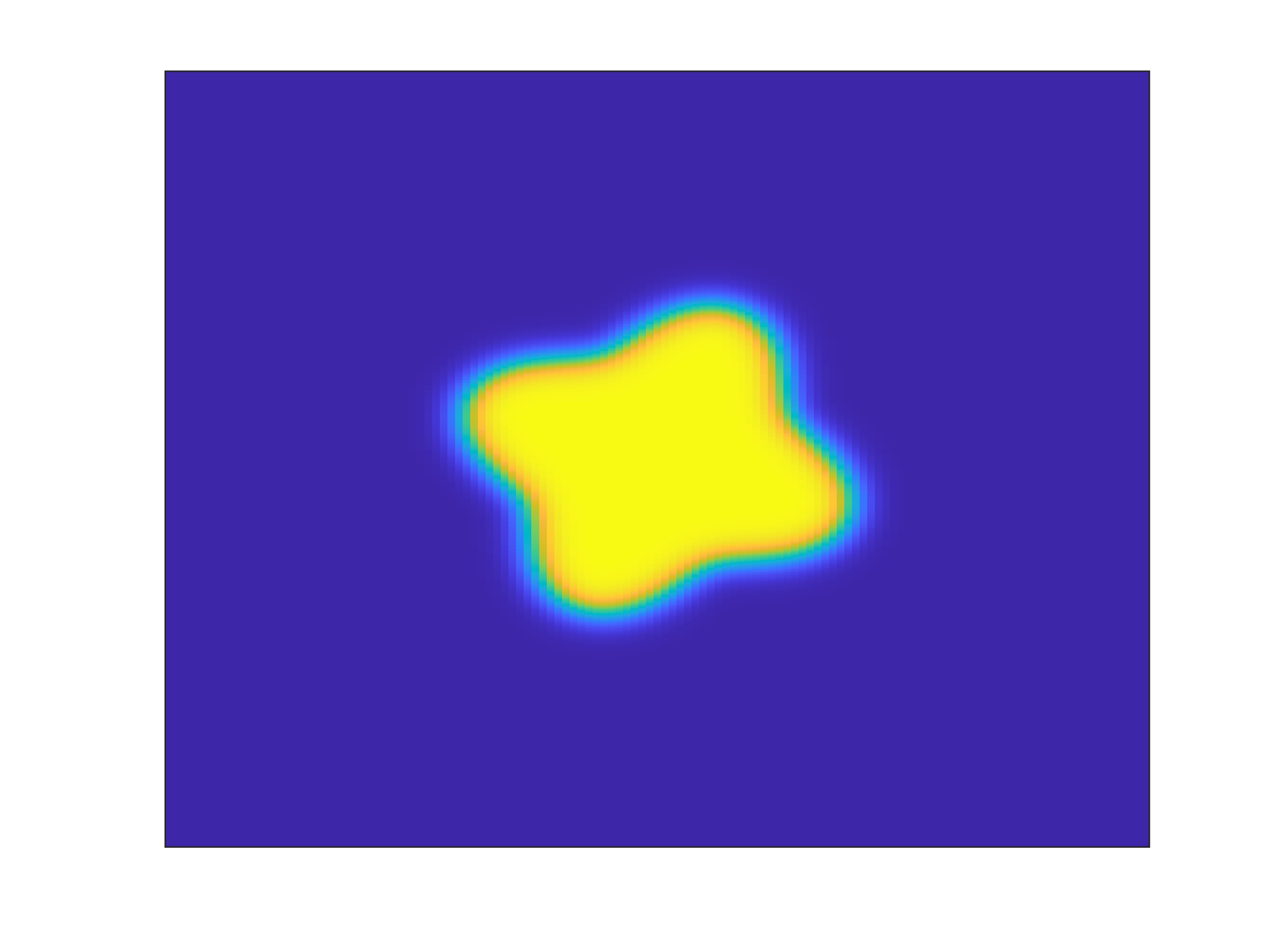}
	}\hspace{-0.8cm}
	\subfigure{
		\label{fig5-3}
		\centering
		\includegraphics[width = 130pt,height=110pt]{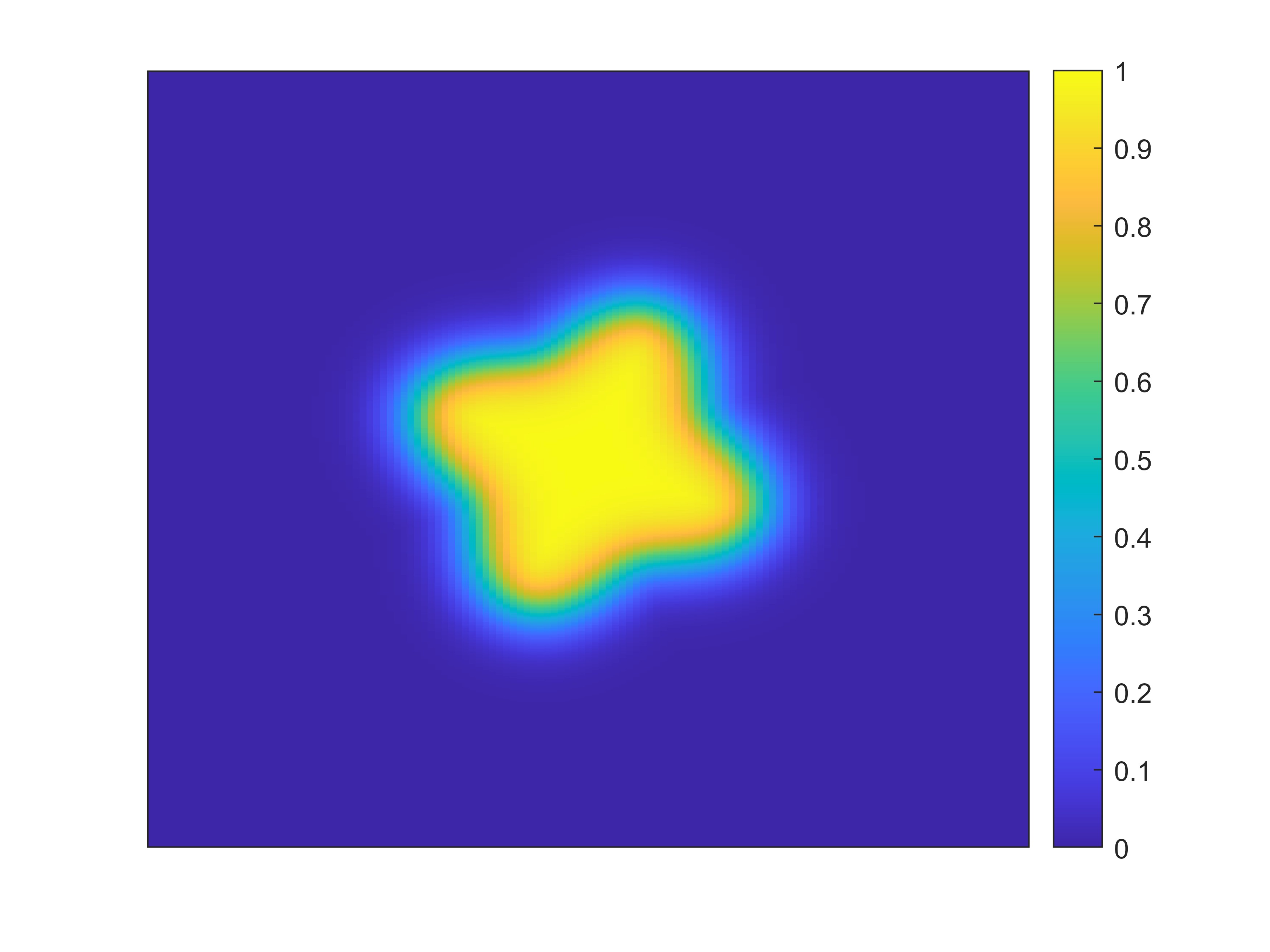}
	}
	\caption{The phase structures of the numerical errors at the terminal time $T$ produced by the DDM approach with interface thickness $\epsilon=1/32,1/16,1/8$ (from left to right) for Example \ref{ex3} in the flower-shaped domain.}
	\label{fig 5}
\end{figure}

\section{Conclusions}\label{conclusion}
In this paper, we studied the approximation error of the diffuse domain method for solving a class of semilinear parabolic equations, imposed with Neumann boundary conditions defined in general irregular domains. Optimal error estimates with respect to weighted $H^1$ and $L^2$ norms are successfully derived. Some numerical examples including ill-posed Allen-Cahn equations are presented to verify the theoretical results. The numerical method and corresponding error analysis framework developed in this paper also naturally enable us to further investigate adapted finite element methods \cite{Ciarlet1978,BinevDahmen2004,Bieterman1982} for solving PDEs in regions with corner points with solid theoretical support.

\bibliographystyle{abbrv}
\bibliography{ref}

\begin{thebibliography}{10}

\bibitem{Abels2009}
Helmut Abels.
\newblock On a diffuse interface model for two-phase flows of viscous,
  incompressible fluids with matched densities.
\newblock 194(2):463--506, 2009.

\bibitem{AbelsLengeler2014}
Helmut Abels and Daniel Lengeler.
\newblock On sharp interface limits for diffuse interface models for two-phase
  flows.
\newblock {\em Interfaces Free Bound.}, 16(3):395--418, 2014.

\bibitem{Adams1975}
Robert~A. Adams.
\newblock {\em Sobolev spaces}, volume Vol. 65 of {\em Pure and Applied
  Mathematics}.
\newblock Academic Press [Harcourt Brace Jovanovich, Publishers], New
  York-London, 1975.

\bibitem{AlandLowengrub2010}
S.~Aland, J.~Lowengrub, and A.~Voigt.
\newblock Two-phase flow in complex geometries: a diffuse domain approach.
\newblock {\em CMES Comput. Model. Eng. Sci.}, 57(1):77--107, 2010.

\bibitem{Acta1979}
Sam Allen and John Cahn.
\newblock A microscopic theory for antiphase domain boundary motion and its
  application to antiphase domain coarsening.
\newblock {\em Acta Metallurgica}, 27:1085--1095, 06 1979.

\bibitem{AndersonMcFadden1998}
D.~M. Anderson, G.~B. McFadden, and A.~A. Wheeler.
\newblock Diffuse-interface methods in fluid mechanics.
\newblock In {\em Annual review of fluid mechanics, {V}ol. 30}, volume~30 of
  {\em Annu. Rev. Fluid Mech.}, pages 139--165. Annual Reviews, Palo Alto, CA,
  1998.

\bibitem{AmlanRay2023}
Amlan~K. Barua, Ray Chew, Shuwang Li, John Lowengrub, Andreas M\"unch, and
  Barbara Wagner.
\newblock Sharp-interface problem of the {O}hta-{K}awasaki model for symmetric
  diblock copolymers.
\newblock {\em J. Comput. Phys.}, 481:Paper No. 112032, 23, 2023.

\bibitem{BedrossianZhu2010}
Jacob Bedrossian, James~H. von Brecht, Siwei Zhu, Eftychios Sifakis, and
  Joseph~M. Teran.
\newblock A second order virtual node method for elliptic problems with
  interfaces and irregular domains.
\newblock {\em J. Comput. Phys.}, 229(18):6405--6426, 2010.

\bibitem{Bieterman1982}
M.~Bieterman and I.~Babu\v~ska.
\newblock The finite element method for parabolic equations. {I}. {A}
  posteriori error estimation.
\newblock {\em Numer. Math.}, 40(3):339--371, 1982.

\bibitem{BinevDahmen2004}
Peter Binev, Wolfgang Dahmen, and Ron DeVore.
\newblock Adaptive finite element methods with convergence rates.
\newblock {\em Numer. Math.}, 97(2):219--268, 2004.

\bibitem{BrannickLiu2015}
J.~Brannick, C.~Liu, T.~Qian, and H.~Sun.
\newblock Diffuse interface methods for multiple phase materials: an energetic
  variational approach.
\newblock {\em Numer. Math. Theory Methods Appl.}, 8(2):220--236, 2015.

\bibitem{BuenoFenton2006}
Alfonso Bueno-Orovio, V\'ictor~M. P\'erez-Garc\'ia, and Flavio~H. Fenton.
\newblock Spectral methods for partial differential equations in irregular
  domains: the spectral smoothed boundary method.
\newblock {\em SIAM J. Sci. Comput.}, 28(3):886--900, 2006.

\bibitem{MartinaBoris2023}
Martina Buka\v~c, Boris Muha, and Abner~J. Salgado.
\newblock Analysis of a diffuse interface method for the {S}tokes-{D}arcy
  coupled problem.
\newblock {\em ESAIM Math. Model. Numer. Anal.}, 57(5):2623--2658, 2023.

\bibitem{BurgerElvetun2015}
Martin Burger, Ole L\o~seth Elvetun, and Matthias Schlottbom.
\newblock Diffuse interface methods for inverse problems: case study for an
  elliptic {C}auchy problem.
\newblock {\em Inverse Problems}, 31(12):125002, 28, 2015.

\bibitem{BurgerElvetun2017}
Martin Burger, Ole L\o~seth Elvetun, and Matthias Schlottbom.
\newblock Analysis of the diffuse domain method for second order elliptic
  boundary value problems.
\newblock {\em Found. Comput. Math.}, 17(3):627--674, 2017.

\bibitem{Ciarlet1978}
Philippe~G. Ciarlet.
\newblock {\em The {F}inite {E}lement {M}ethod for {E}lliptic {P}roblems}.
\newblock North-Holland Publishing Co., Amsterdam-New York-Oxford, 1978.

\bibitem{DolbowHarari2009}
John Dolbow and Isaac Harari.
\newblock An efficient finite element method for embedded interface problems.
\newblock {\em Internat. J. Numer. Methods Engrg.}, 78(2):229--252, 2009.

\bibitem{DuFeng2020}
Qiang Du and Xiaobing Feng.
\newblock The phase field method for geometric moving interfaces and their
  numerical approximations.
\newblock In {\em Geometric partial differential equations. {P}art {I}},
  volume~21 of {\em Handb. Numer. Anal.}, pages 425--508.
  Elsevier/North-Holland, Amsterdam, 2020.

\bibitem{ElliottStinner2011}
Charles~M. Elliott, Bj\"orn Stinner, Vanessa Styles, and Richard Welford.
\newblock Numerical computation of advection and diffusion on evolving diffuse
  interfaces.
\newblock {\em IMA J. Numer. Anal.}, 31(3):786--812, 2011.

\bibitem{FeireislRocca2010}
Eduard Feireisl, Hana Petzeltov\'a, Elisabetta Rocca, and Giulio Schimperna.
\newblock Analysis of a phase-field model for two-phase compressible fluids.
\newblock {\em Math. Models Methods Appl. Sci.}, 20(7):1129--1160, 2010.

\bibitem{FranzRoos2012}
Sebastian Franz, Roland G\"artner, Hans-G\"org Roos, and Axel Voigt.
\newblock A note on the convergence analysis of a diffuse-domain approach.
\newblock {\em Comput. Methods Appl. Math.}, 12(2):153--167, 2012.

\bibitem{FriesBelytschko2010}
Thomas-Peter Fries and Ted Belytschko.
\newblock The extended/generalized finite element method: an overview of the
  method and its applications.
\newblock {\em Internat. J. Numer. Methods Engrg.}, 84(3):253--304, 2010.

\bibitem{FrigeriGrasselli2015}
Sergio Frigeri, Maurizio Grasselli, and Elisabetta Rocca.
\newblock A diffuse interface model for two-phase incompressible flows with
  non-local interactions and non-constant mobility.
\newblock {\em Nonlinearity}, 28(5):1257--1293, 2015.

\bibitem{Roger1984}
David~F. Griffiths, editor.
\newblock {\em The {M}athematical {B}asis of {F}inite {E}lement {M}ethods},
  volume 343.
\newblock Clarendon Press, Oxford, 1984.

\bibitem{GuoYu2021}
Zhenlin Guo, Fei Yu, Ping Lin, Steven Wise, and John Lowengrub.
\newblock A diffuse domain method for two-phase flows with large density ratio
  in complex geometries.
\newblock {\em J. Fluid Mech.}, 907:Paper No. A38, 28, 2021.

\bibitem{HaoJu2025}
Wenrui Hao, Lili Ju, and Yuejin Xu.
\newblock Optimal error estimates of the diffuse domain method for second order
  parabolic equations, 2025.

\bibitem{HellrungLee2012}
Jeffrey~Lee Hellrung, Jr., Luming Wang, Eftychios Sifakis, and Joseph~M. Teran.
\newblock A second order virtual node method for elliptic problems with
  interfaces and irregular domains in three dimensions.
\newblock {\em J. Comput. Phys.}, 231(4):2015--2048, 2012.

\bibitem{Jerg2020}
Katharina~I. Jerg, Ren\'e{}~Phillip Austerm\"uhl, Karsten Roth, Jonas
  Gro\ss~e{} Sundrup, Guido Kanschat, J\"urgen~W. Hesser, and Lisa Wittmayer.
\newblock Diffuse domain method for needle insertion simulations.
\newblock {\em Int. J. Numer. Methods Biomed. Eng.}, 36(9):e3377, 18, 2020.

\bibitem{KockelkorenLevine2003}
Julien Kockelkoren, Herbert Levine, and Wouter-Jan Rappel.
\newblock Computational approach for modeling intra- and extracellular
  dynamics.
\newblock {\em Phys. Rev. E}, 68:037702, Sep 2003.

\bibitem{KarlLowengrub2015}
Karl~Yngve Lerv\aa~g and John Lowengrub.
\newblock Analysis of the diffuse-domain method for solving {PDE}s in complex
  geometries.
\newblock {\em Commun. Math. Sci.}, 13(6):1473--1500, 2015.

\bibitem{LeVequeLi1995}
Randall~J. LeVeque and Zhi~Lin Li.
\newblock Erratum: ``{T}he immersed interface method for elliptic equations
  with discontinuous coefficients and singular sources'' [{SIAM} {J}. {N}umer.\
  {A}nal.\ {\bf 31} (1994), no.\ 4, 1019--1044; {MR}1286215 (95g:65139)].
\newblock {\em SIAM J. Numer. Anal.}, 32(5):1704, 1995.

\bibitem{LiFeng2015}
Jian~Jing Li and Xiu~Fang Feng.
\newblock Higher-order finite difference scheme for solving one-dimensional
  elliptic equations with discontinuous coefficient and singular sources.
\newblock {\em Commun. Appl. Math. Comput.}, 29(4):503--513, 2015.

\bibitem{LiLowengrub2009}
X.~Li, J.~Lowengrub, A.~R\"atz, and A.~Voigt.
\newblock Solving {PDE}s in complex geometries: a diffuse domain approach.
\newblock {\em Commun. Math. Sci.}, 7(1):81--107, 2009.

\bibitem{LiIto2006}
Zhilin Li and Kazufumi Ito.
\newblock {\em The immersed interface method}, volume~33 of {\em Frontiers in
  Applied Mathematics}.
\newblock Society for Industrial and Applied Mathematics (SIAM), Philadelphia,
  PA, 2006.
\newblock Numerical solutions of PDEs involving interfaces and irregular
  domains.

\bibitem{LiuChai2022}
Xi~Liu, Zhenhua Chai, Chengjie Zhan, Baochang Shi, and Wenhuan Zhang.
\newblock A diffuse-domain phase-field lattice {B}oltzmann method for two-phase
  flows in complex geometries.
\newblock {\em Multiscale Model. Simul.}, 20(4):1411--1436, 2022.

\bibitem{NguyenStoter2018}
Lam~H. Nguyen, Stein K.~F. Stoter, Martin Ruess, Manuel~A. Sanchez~Uribe, and
  Dominik Schillinger.
\newblock The diffuse {N}itsche method: {D}irichlet constraints on phase-field
  boundaries.
\newblock {\em Internat. J. Numer. Methods Engrg.}, 113(4):601--633, 2018.

\bibitem{AndreasAxel2006}
Andreas R\"atz and Axel Voigt.
\newblock P{DE}'s on surfaces---a diffuse interface approach.
\newblock {\em Commun. Math. Sci.}, 4(3):575--590, 2006.

\bibitem{Schlottbom2016}
Matthias Schlottbom.
\newblock Error analysis of a diffuse interface method for elliptic problems
  with {D}irichlet boundary conditions.
\newblock {\em Appl. Numer. Math.}, 109:109--122, 2016.

\bibitem{Stoter2017}
Stein K.~F. Stoter, Peter M\"uller, Luca Cicalese, Massimiliano Tuveri, Dominik
  Schillinger, and Thomas J.~R. Hughes.
\newblock A diffuse interface method for the {N}avier-{S}tokes/{D}arcy
  equations: perfusion profile for a patient-specific human liver based on
  {MRI} scans.
\newblock {\em Comput. Methods Appl. Mech. Engrg.}, 321:70--102, 2017.

\bibitem{TeigenLi2009}
Knut~Erik Teigen, Xiangrong Li, John Lowengrub, Fan Wang, and Axel Voigt.
\newblock A diffuse-interface approach for modeling transport, diffusion and
  adsorption/desorption of material quantities on a deformable interface.
\newblock {\em Commun. Math. Sci.}, 7(4):1009--1037, 2009.

\bibitem{TorabiLowengrub2009}
Solmaz Torabi, John Lowengrub, Axel Voigt, and Steven Wise.
\newblock A new phase-field model for strongly anisotropic systems.
\newblock {\em Proc. R. Soc. Lond. Ser. A Math. Phys. Eng. Sci.},
  465(2105):1337--1359, 2009.
\newblock With supplementary material available online.

\bibitem{WangChertock2023}
Chenxi Wang, Alina Chertock, Shumo Cui, Alexander Kurganov, and Zhen Zhang.
\newblock A diffuse-domain-based numerical method for a chemotaxis-fluid model.
\newblock {\em Math. Models Methods Appl. Sci.}, 33(2):341--375, 2023.

\bibitem{WangKai2023}
Wang Xiao, Kai Liu, John Lowengrub, Shuwang Li, and Meng Zhao.
\newblock Three-dimensional numerical study on wrinkling of vesicles in
  elongation flow based on the immersed boundary method.
\newblock {\em Phys. Rev. E}, 107(3):Paper No. 035103, 12, 2023.

\bibitem{YangMao2019}
Jinjin Yang, Shipeng Mao, Xiaoming He, Xiaofeng Yang, and Yinnian He.
\newblock A diffuse interface model and semi-implicit energy stable finite
  element method for two-phase magnetohydrodynamic flows.
\newblock {\em Comput. Methods Appl. Mech. Engrg.}, 356:435--464, 2019.

\bibitem{ZhangEzhov2025}
Ray~Zirui Zhang, Ivan Ezhov, Michal Balcerak, Andy Zhu, Benedikt Wiestler,
  Bjoern Menze, and John~S. Lowengrub.
\newblock Personalized predictions of glioblastoma infiltration: Mathematical
  models, physics-informed neural networks and multimodal scans.
\newblock {\em Medical Image Analysis}, 101:103423, 2025.

\bibitem{ZhaoWei2009}
Shan Zhao and G.~W. Wei.
\newblock Matched interface and boundary ({MIB}) for the implementation of
  boundary conditions in high-order central finite differences.
\newblock {\em Internat. J. Numer. Methods Engrg.}, 77(12):1690--1730, 2009.

\end{thebibliography}
\end{document}